\documentclass[a4paper,11pt,oneside,reqno]{amsart}

\usepackage[authoryear,longnamesfirst]{natbib}

\def\cite{\citet}

\numberwithin{equation}{section}
\usepackage{showlabels}

\RequirePackage{bbm}
\RequirePackage{mathrsfs}
\RequirePackage{amsmath,amsfonts,amssymb,amsthm}

\usepackage{mdwtab}

\numberwithin{equation}{section}
\allowdisplaybreaks[4]

\theoremstyle{plain}
\newtheorem{theorem}{Theorem}[section]

\newtheorem{lemma}[theorem]{Lemma}
\newtheorem{corollary}[theorem]{Corollary}

\theoremstyle{definition}
\newtheorem{definition}{Definition}[section]
\newtheorem{example}[definition]{Example}
\newtheorem{remark}[definition]{Remark}

\newcounter{constr}
\newcounter{subconstr}

\makeatletter

\newenvironment{constr}
{\smallskip\stepcounter{constr}\setcounter{subconstr}{1}%
\def\@currentlabel{\arabic{constr}\Alph{subconstr}}%
\noindent\textbf{Construction~\@currentlabel.} \begingroup\it}
{\endgroup\medskip}

\newenvironment{subconstr}
{\smallskip\stepcounter{subconstr}%
\def\@currentlabel{\arabic{constr}\Alph{subconstr}}%
\noindent\textbf{Construction~\@currentlabel.} \begingroup\it}
{\endgroup\medskip}

\makeatother

\newcommand{\IZ}{\mathbbm{Z}}
\newcommand{\IR}{\mathbbm{R}}

\def\%#1{\mathcal{#1}}
\newcommand{\law}{\mathscr{L}}
\newcommand{\eqlaw}{\stackrel{\mathscr{D}}{=}}

\def\ER{Erd\H{o}s-R\'enyi}

\def\dw{\mathop{d_{\mathrm{W}}}}
\def\dk{\mathop{d_{\mathrm{K}}}}

\newcommand\indep{\protect\mathpalette{\protect\independenT}{\perp}}
\def\independenT#1#2{\mathrel{\rlap{$#1#2$}\mkern2mu{#1#2}}}

\newcommand{\bigo}{\mathrm{O}}

\newcommand{\ssum}{\mathop{\textstyle\sum}}

\newcommand{\nsig}{\Sigma\kern-0.5em\raise0.2ex\hbox to 0pt{$\mid$}\kern0.5em}
\def\tsfrac#1#2{{\textstyle\frac{#1}{#2}}}
\newcommand{\toinf}{\to\infty}
\newcommand{\tozero}{\to0}

\newcommand{\eps}{\varepsilon}
\renewcommand{\phi}{\varphi}
\newcommand{\D}{\Delta}

\newcommand{\I}{\mathrm{I}}

\newcommand{\ahalf}{{\textstyle\frac{1}{2}}}

\newcommand{\anth}{{\textstyle\frac{1}{n}}}

\newcommand{\eq}{\eqref}
\newcommand{\IE}{\mathbbm{E}}
\newcommand{\IP}{\mathbbm{P}}
\newcommand{\Var}{\mathop{\mathrm{Var}}\nolimits}
\newcommand{\Cov}{\mathop{\mathrm{Cov}}}

\newcommand{\Vol}{\mathop{\mathrm{Vol}}}

\newcommand{\Bi}{\mathop{\mathrm{Bi}}}

\newcommand{\Be}{\mathop{\mathrm{Be}}}
\newcommand{\N}{\mathrm{N}}

\def\be#1\ee{\begin{linenomath}\begin{equation*}#1%
\end{equation*}\end{linenomath}}
\def\ben#1\ee{\begin{linenomath}\begin{equation}#1%
\end{equation}\end{linenomath}}

\def\bes#1\ee{\begin{linenomath}\begin{equation*}\begin{split}#1%
\end{split}\end{equation*}\end{linenomath}}
\def\besn#1\ee{\begin{linenomath}\begin{equation}\begin{split}#1%
\end{split}\end{equation}\end{linenomath}}

\def\bg#1\ee{\begin{linenomath}\begin{gather*}#1%
\end{gather*}\end{linenomath}}
\def\bgn#1\ee{\begin{linenomath}\begin{gather}#1%
\end{gather}\end{linenomath}}

\def\bm#1\ee{\begin{linenomath}\begin{multline*}#1%
\end{multline*}\end{linenomath}}
\def\bmn#1\ee{\begin{linenomath}\begin{multline}#1%
\end{multline}\end{linenomath}}

\def\ba#1\ee{\begin{linenomath}\begin{align*}#1%
\end{align*}\end{linenomath}}
\def\ban#1\ee{\begin{linenomath}\begin{align}#1%
\end{align}\end{linenomath}}

\def\klr#1{(#1)}
\def\bklr#1{\bigl(#1\bigr)}
\def\bbklr#1{\Bigl(#1\Bigr)}
\def\bbbklr#1{\biggl(#1\biggr)}

\def\bklrl{\bigl(}

\def\bklrr{\bigr)}

\def\kle#1{[#1]}
\def\bkle#1{\bigl[#1\bigr]}
\def\bbkle#1{\Bigl[#1\Bigr]}

\def\bklel{\bigl[}
\def\bbklel{\Bigl[}

\def\bkler{\bigr]}
\def\bbkler{\Bigr]}

\def\klg#1{\{#1\}}
\def\bklg#1{\bigl\{#1\bigr\}}

\def\bbbklg#1{\biggl\{#1\biggr\}}

\def\norm#1{\Vert#1\Vert}

\def\abs#1{\vert#1\vert}
\def\babs#1{\bigl\vert#1\bigr\vert}
\def\bbabs#1{\Bigl\vert#1\Bigr\vert}
\def\bbbabs#1{\biggl\vert#1\biggr\vert}

\def\mid{\vert}
\def\bmid{\bigm\vert}
\def\bbmid{\Bigm\vert}

\def\^#1{\ifmmode {\mathaccent"705E #1} \else {\accent94 #1} \fi}
\def\~#1{\ifmmode {\mathaccent"707E #1} \else {\accent"7E #1} \fi}
\def\*#1{#1^\ast}
\edef\-#1{\noexpand\ifmmode {\noexpand\bar{#1}} \noexpand\else \-#1\noexpand\fi}
\def\>#1{\vec{#1}}
\def\.#1{\dot{#1}}

\def\leq{\leqslant}
\def\geq{\geqslant}

\makeatletter
\def\atop{\@@atop}
\makeatother

\newcount\minute
\newcount\hour
\newcount\hourMins
\def\now{%
\minute=\time%
\hour=\time \divide \hour by 60%
\hourMins=\hour \multiply\hourMins by 60%
\advance\minute by -\hourMins%
\zeroPadTwo{\the\hour}:\zeroPadTwo{\the\minute}%
}
\def\zeroPadTwo#1{\ifnum #1<10 0\fi#1}

\def\blfootnote{\xdef\@thefnmark{}\@footnotetext}

\def\phi{\varphi}
\def\^#1{\mathaccent"705E #1}
\def\~#1{\mathaccent"707E #1}
\def\trunc#1{\-{#1}}

\usepackage[mathlines,modulo]{lineno}
\begin{document}
\title{Stein couplings for normal approximation}
\author{Louis H. Y. Chen \and Adrian R\"ollin}
\thanks{Version from \today, \now}

\begin{abstract} In this article we propose a general framework for normal
approximation using Stein's method. We introduce the new concept of \emph{Stein
couplings} and we show that it lies at the heart of popular approaches such as
the local approach, exchangeable pairs, size biasing and many other
approaches. We prove several theorems with which normal approximation for the
Wasserstein and Kolmogorov metrics becomes routine once a Stein coupling is
found. To illustrate the versatility of our framework we give applications in
Hoeffding's combinatorial central limit theorem, functionals in the classic
occupancy scheme, neighbourhood statistics of point patterns with fixed number
of points and functionals of the components of randomly chosen vertices of
sub-critical \ER\ random graphs. In all these cases, we use new, non-standard
couplings.
\end{abstract}
\maketitle

\vspace{-5mm}
{\small
\tableofcontents
}

\section{Introduction}

Since its introduction in the early 70s, Stein's method has gone through a vivid
development. First used by \cite{Stein1972} for normal approximation of
$m$-dependent sequences, it was gradually modified and generalized to other
distributions and different dependency settings. Among the few important
contributions that influenced the theoretical
understanding of the method is the book by \cite{Stein1986}, where the concepts
of \emph{auxiliary randomization} and \emph{exchangeable pairs} are introduced.
Another corner stone is the \emph{generator method}, independently introduced by
\cite{Barbour1988} and \cite{Goetze1991}, which allows for an adaptation of the
method to more complicated approximating distributions, such as compound Poisson
distribution, Poisson point processes and Gaussian diffusions.
\cite{Diaconis1991} give a connection between Stein's method and orthogonal
polynomials; see also \cite{Goldstein2005}, who related this approach to
distributional transformations. A more recent development was initiated by
\cite{Nourdin2009} and related articles, where a fruitful theory for normal
approximation for functionals of Gaussian measures, Rademacher sequences and
Poisson measures is developed. Despite these achievements,
relatively little effort has been put into building up a rigorous theoretical
framework for the method in order to unify and generalize the variety of known
results. 

In particular for normal approximation, a wide range of different approaches has
appeared over the last decades. Among the most prominent approaches are the
\emph{local approach}, dating back to \cite{Stein1972} and extensively studied
by \cite{Chen2004a}, the \emph{exchangeable pairs approach} by
\cite{Stein1986}, further developed by \cite{Rinott1997},
\cite{Rollin2007a} and \cite{Rollin2008} with
a variety of applications such as weighted $U$-statistics, anti-voter model on
finite graphs and models in statistical mechanics, and the \emph{size} and
\emph{zero bias couplings} by
\cite{Goldstein1996} and \cite{Goldstein1997}, respectively. Another approach
was introduced by \cite{Chatterjee2008} for functionals of independent
random variables (we will discuss a more general version of this
approach---called \emph{interpolation to independence}---in Section~\ref{sec10}).
In addition to these abstract approaches, many other ad-hoc constructions have
been used to tackle specific problems, such as \cite{Ho1978} and
\cite{Bolthausen1984} for the combinatorial limit theorem, \cite{Barbour1989}
and \cite{Rollin2008a} for refined versions of local dependence and
\cite{Barbour1986} and \cite{Zhao1997} for double indexed permutation
statistics. However, despite all these achievements, a unifying framework is
still missing and connections between the different approaches given in the
literature are vague, at best. Although making an attempt to systematically
discuss Stein's method, survey articles such as \cite{Reinert1998a} and
\cite{Rinott2000} illustrate the key issue here: for each approach a separate
theorem is proved, and then typically only for one specific metric. 

The reason for this seems to be the following. For all of these approaches,
the involved
random variables either have to satisfy some more or less abstract conditions or
some specific properties in the dependence structure are exploited to obtain the
results. This could be a defining equation like in the size-biasing
approach, a linearity assumption on a conditional expectation for
exchangeable pairs, a local dependence structure, or other properties and
conditions. Depending on the specific form of these conditions, the quantities
arising from using Stein's method---seemingly---have to be handeled differently.
Although in simpler applications it might be clear how to directly manipulate
these expressions in an ad-hoc way in order to successfully apply Stein's
method, this is less feasible for more complex situations. \cite{Chatterjee2008}
proposes an approach for functionals of finite collections of independent random
variables. His approach comes at the cost of a rather complicated bound, and for
many applications optimal results (in terms of moment conditions and metrics)
may not be obtained that way (c.f.\ the difference between
Corollaries~\ref{cor1} and~\ref{cor2} below).  It is crucial to exploit
properties of the random variables at hand in order to express the error bounds
in terms of simple and manageable expressions. To achieve this, we will
explore what the abstract key conditions are that allow for a successful
implementation of Stein's method for normal approximation---a question that has
not yet been addressed in the literature. Our framework provides such a set of
conditions along with a variety of ``plug-in'' type theorems. Not only does our
framework show the connection between all the above mentioned approaches, but
it also introduces some crucial generalisations and hence flexibility into
Stein's method for normal approximation. 

Our main tool is that of couplings; more specifically, we introduce the new
concept of \emph{Stein couplings}. We provide general approximation theorems
with respect to the Wasserstein and Kolmogorov metrics, where the error terms
are expressed in terms of the relationship between the involved coupled random
variables. Although we relate the different known approaches via Stein
couplings, our applications also show that the distinction usually made between
these approaches is rather artificial. Most of the couplings used in our
applications cannot be clearly assigned to one of the known approaches, but
emerge naturally from the problem at hand and are therefore constructed in a
more ad-hoc way. Nevertheless, in Section~\ref{sec3} we make an attempt to give
a systematic overview over the different coupling constructions, being very
well aware of the fact that other ways of classifying these couplings may
be equally reasonable.

\subsection{A short introduction to Stein's method}

We assume throughout this article that $W$ is a random variable whose
distribution is to be approximated by a standard normal distribution and we also
assume that $W$ has finite variance. Stein's method for normal approximation is
based on the fact that, for all, say, Lipschitz continuous function $f$ we
have 
\ben                                                        \label{1}
  \IE\klg{Zf(Z)} = \IE f'(Z)
\ee
if and only if $Z\sim\N(0,1)$. Now, if it is the case that
\ben                                                        \label{2}
  \IE\klg{Wf(W)} \approx \IE f'(W)
\ee
for many functions $f$, we would expect that $W$ is close to the normal
distribution. With Stein's method we can make this heuristic idea rigorous 
in the following sense. Assume that, in order to measure the closeness of
$\law(W)$ and $\law(Z)$, we would like to bound
\ben                                                      \label{3}
  \IE h(W) - \IE h(Z)
\ee
for some function~$h$ (take for example the half line indicators
for
the Kolmogorov metric). Solving the so-called \emph{Stein equation}
\ben                                                            \label{4}      
  f'(x)-xf(x) = h(x)-\IE h(Z)
\ee
for $f=f_h$, we can then express the quantity \eq{3} as
\ben                                                               \label{5}
  \IE h(W) - \IE h(Z) = \IE\bklg{f'(W)-Wf(W)} = \IE Af(W),
\ee
where $A$ is the operator defined by $Af(x) := f'(x)-xf(x)$. Hence, $\IE Af(W)$
measures the error in \eq{2} and the function $f$ relates this error to
the error of the approximation $\IE h(W) \approx \IE h(Z)$ via \eq{5} (see
\cite{Rollin2007} for a more detailed discussion).

Let us elaborate \eq{2} more rigorously. One way to express \eq{2} is to
assume that there are two random variables $T_1$ and $T_2$ such that 
\ben                                                            \label{6}
      \IE\bklg{Wf(W)} = \IE\bklg{T_1 f'(W+T_2)}
\ee
for all~$f$. Equations of the form \eq{6} are often called \emph{Stein
identities}, as they characterise in some sense the distribution of~$W$.
There are two important special cases of \eq{6}. If, for example,
$\Psi$ is a Gaussian field and $W=W(\Psi)$ a (smooth) functional of it,
\cite{Nourdin2009} use Maliavin calculus to derive \eq{6} with $T_2=0$ and they
give a more or less explicit expression for~$T_1$. In contrast,
in the zero-biasing approach as introduced by \cite{Goldstein1997}, it is
assumed
that \eq{6} holds for a specific $T_2$ where~$T_1 = 1$.

To illustrate the line of argument to obtain a final bound in its simplest form,
let us look at the case~$T_2=0$. We can write
\ben                                                                 \label{7}
    \IE A f(W) = \IE\bklg{(1-T_1)f'(W)}  = \IE\bklg{(1-\IE^W T_1)f'(W)},
\ee
so that the error in the normal approximation is given by 
\be
  \abs{\IE h(W)-\IE h(Z)} = \abs{\IE Af(W)} \leq \norm{f'}\IE\abs{1-\IE^W T_1},
\ee
and, if $\IE T_1 =1$, the last term is usually further bounded by
$\sqrt{\Var\IE^W T_1}$. It is
not difficult to show that $\norm{f'}\leq 2\norm{h}$ (where
$\norm{\cdot}$ denotes the supremum norm). Hence, we obtain the final
bound
\be
  \abs{\IE h(W)-\IE h(Z)} \leq 2\norm{h}\sqrt{\Var \IE^W T_1}.
\ee

Note that this specific form of the bound involving $\Var \IE^W T_1$ has been
implicitly used in the literature around Stein's method many times, but probably
first made explicit by \cite{Cacoullos1994}. In \cite{Chatterjee2009} and
\cite{Nourdin2009a} a connection with Poincar\'e
inequalities was made, where a bound on $\Var\IE^W T_1$ is called a
\emph{Poincar\'e inequality of second order}.

\subsection{An outline of our approach}\label{sec1}

Roughly speaking, we propose a general, probabilistic method of deriving
identities of the form \eq{6} and more refined versions of it, where we will
make use of both random variables $T_1$ and $T_2$ as we will show below. Based
on this, we provide theorems to obtain bounds for the accuracy of a normal
approximation of~$W$. 

The key idea is that of \emph{auxiliary randomisation}, introduced by
\cite{Stein1986}. To this end we construct a random variable $W'$ which we
imagine to be a `small perturbation' of~$W$. It is important to emphasize that
we make no assumptions about the distribution of $W'$ or the joint distribution
at this point (in particular no exchangeability is assumed a priori). We also
need to note that we never attempt to couple $W$ and $W'$ so that $W'=W$ almost
surely; in fact, such a coupling would contain no useful information for our
purposes. It is crucial that \emph{some} randomness remains (see \cite{Ross2008}
for a discussion on how to optimally choose the perturbation in some examples
using exchangeable pairs) and it will become clear that the difference $D:=W'-W$
contains essential information about~$W$. This can be seen as a
\emph{local-to-global} approach, where we deduce global properties of $W$ from
the behaviour of local perturbations, as these are often easier to handle.

When dealing with Stein's method, it becomes clear that we cannot expect normal
approximation results from any arbitrary coupling $(W,W')$, and we need to
impose some structure. To achieve this, we generalize an idea which goes back to
\cite{Stein1972}, and introduce a third random variable~$G$. For reasons that
will hopefully become apparent in the course of this article, we then make the
following key definition.

\begin{definition}\label{def} Let $(W,W',G)$ be a coupling of square integrable
random variables. We call $(W,W',G)$ a \emph{Stein coupling} if 
\ben                                                                \label{8}
    \IE\bklg{Gf(W') - Gf(W)} = \IE\bklg{Wf(W)}
\ee
for all functions for which the expectations exist. 
\end{definition}

Before explaining how this will help in finding an identity of the
form~\eq{6}, let us first discuss  some standard Stein couplings. As a simple
example where \eq{8} holds, assume that $(W,W')$ is an exchangeable
pair and assume that, for some $\lambda>0$ we have
\ben                                                                \label{9}
    \IE^W(W'-W) = -\lambda W.
\ee
If we set $G=\frac{1}{2\lambda}(W'-W)$ it is easy to see that \eq{8} is
satisfied (see Section~\ref{sec4} for more details).
Equation~\eq{9} is the well-known linear regression condition introduced by
Stein in \cite{Diaconis1977} and \cite{Stein1986} and \eq{8} can be seen as
a generalization of~it. We will show in Section~\ref{sec3} that~\eq{8} is the
key to normal approximation using Stein's method and that many approaches in the
literature in fact (implicitly) establish~\eq{8}. 

\begin{remark}\label{rem1}
Let $(W,W',G)$ be a Stein coupling. If we choose $f(x)= 1$, we see from
\eq{8} that~$\IE W = 0$. If we choose $f(x)=x$ we furthermore have that $\IE(GD)
= \Var W$.
\end{remark}
Note that the statement $\IE(GD) = \Var W$ is well known in a special case.
If $W'$ is an independent copy of $W$, then $(W,W',(W'-W)/2)$ is a
Stein coupling (use the exchangeable pairs approach above with $\lambda=1$) and,
hence,
\be
  \Var W = \ahalf\IE(W-W')^2 = \IE(GD)
\ee
is the well known way to express the variance of $W$ in terms of two independent
copies. However, this coupling is not useful for our purpose as $\abs{W'-W}$ is
not `small'. This example also shows that a Stein coupling by itself does by no
means guarantee proximity to the normal distribution.

It is often not difficult to construct a `small perturbation' $W'$ of~$W$. Three
main techniques have been used in the literature: deletion,
replacement and
duplication (note, however, that this is often done implicitly and not
expressed in terms of couplings).  In many typical situations, $W$ is a
functional of a family of random variables $X_1,\dots,X_n$ and $W'$ can be
constructed by picking a random index $I$ independently of the $X_i$ and then by
perturbing at position~$I$, either by removing $X_I$, replacing it, or by adding
another, related random variable. If the $X_i$ are not independent, the
other random variables typically have to be `adjusted' appropriately. Let
us quickly illustrate the three techniques in the most simple situation, that of
a sum of independent random variables. Let $W = \sum_{i=1}^n X_i$, where the
$X_i$ are assumed to be independent, centred and such that $\Var X_i = 1/n$.
Let in what follows $I$ be uniformly distributed over $\{1,\dots,n\}$ and
independent of all else.

\smallskip

\paragraph{\it Deletion} Define $W' = W - X_I$, that is, remove $X_I$ from~$W$.
If we choose $G = - n X_I$, we have
\be
    \IE\bklg{G f(W')} = -\sum_{i=1}^n \IE\bklg{X_i f(W-X_i)} = 0
\ee
due to the independence assumption. Further,
\be
  -\IE\bklg{G f(W)} = \sum_{i=1}^n \IE\bklg{X_i f(W)} = \IE\bklg{Wf(W)},
\ee
so that, indeed, \eq{8} is satisfied. This construction is very powerful under
local dependence, but it can also be used in other contexts; see
Section~\ref{sec12}.

\smallskip

\paragraph{\it Replacement} Let $X_1',\dots,X_n'$ be independent copies of
the~$X_i$. Define $W' = W - X_I + X_I'$. Then, it is not difficult to see that
$(W,W')$ is an exchangeable pair and that \eq{9} holds with $\lambda=1/n$,
which corresponds to $G = \frac{n}{2}(X_I'-X_I)$; this implies~\eq{8}. The idea
of replacing (or re-sampling) is one of the most fruitful in
Stein's method and will often
lead to an exchangeable pair $(W,W')$ and can be applied in situations where the
functional is no longer a sum or where some weak global dependence
structure is present. It lends itself naturally if $W$ can be interpreted as the
state of a stationary Markov chain and $W'$ is a step ahead in the chain;
this observation was first made explicit by \cite{Rinott1997}. It is important
to note here that the choice of $G$ is by far not restricted to
be a multiple of $(W'-W)$, also if $(W,W')$ is exchangeable. This added
flexibility is one of the key observation in this article. Note also that the
$X_i'$ need not be copies of the $X_i$ and may as well have a different
distribution, so that $(W,W')$ need not be exchangeable. In the size-biasing
approach, indeed, the $X_i'$ will typically be
chosen to have the size-biased distribution of~$X_i$.

\smallskip

\paragraph{\it Duplication} To the best of our knowledge, this method has been
only used by \cite{Chen1998}. Let the $X_i'$ be as in the previous paragraph,
and let $W' = W + X_I'$ along with $G = n (X_I'-X_I)$. Due to symmetry we have
\be
  \IE\bklg{X_I'f(W')} = \IE\bklg{X_If(W')},
\ee
hence $\IE\bklg{Gf(W')} = 0$. Further, $\IE \bklg{X_I' f(W)} = 0$, 
so that $\IE\bklg{Gf(W)} = \IE\bklg{Wf(W)}$ and \eq{8} follows.

\smallskip

As can be readily seen from this example, there are typically many possible ways
to construct Stein couplings and it depends on the application which
perturbation will give optimal results. Typically, two of the three random
variables $W$, $W'$ and $G$ are easy to construct (usually $(W,W')$ or
$(W,G)$) and the challenge lies then in constructing the third random variable
to make the triple a Stein coupling. However, as we will show in
Section~\ref{sec3}, by making
abstract assumptions about the structure of $W$, many situations can be handled
by standard couplings, so that, in a concrete application, one often only needs
to concentrate on constructing the coupling satisfying some abstract conditions
rather than to find a Stein coupling from scratch---although the latter can
give interesting additional insight into the problem at hand and may also lead
to improved bounds. 

We need to emphasize at this point that our abstract theorems will also
hold if \eq{8} is not satisfied, so that---a priori---any coupling $(W,W',G)$
can be used. However, useful bounds can only be expected if \eq{8} holds at
least approximately and the accuracy at which \eq{8} holds enters explicitly
into our error bounds. This parallels the introduction of a remainder term in
Condition \eq{9} by \citet[Eq.~(1.7)]{Rinott1997}.

Let us now go back and show how a coupling $(W,W',G)$ helps in
obtaining a Stein identity of the form~\eq{6}. By the fundamental theorem of
calculus we have
\ben                                                    \label{11}
    f(W') - f(W) = \int_{0}^{D}f'(W+t)dt,
\ee
so that, multiplying \eq{11} by $G$ and taking expectation, we obtain
\ben                                                    \label{12}
    \IE\bklg{Gf(W')-Gf(W)}= \IE\bbbklg{G\int_{0}^{D}f'(W+t)dt }.
\ee
If $U$ is an independent random variable with uniform distribution on $[0,1]$
we can also write this as 
\ben                                                        \label{13}
  \IE\bklg{Gf(W')-Gf(W)}= \IE\bklg{GDf'(W+UD)}.                   
\ee
If $(W,W',G)$ is a Stein coupling, the left hand side of \eq{13}
equals to $\IE\bklg{Wf(W)}$ and, hence, \eq{6} is satisfied with $T_1=GD$
and~$T_2=UD$. Our generality comes at a cost: as is clear from \eq{6}, if $T_2$
is
non-trivial we can not easily condition $T_1$ on $W$ (or an appropriate larger
$\sigma$-algebra) as done in \eq{7}---an important step in the argument.
However, using a simple Taylor expansion, we will see how to circumvent this
problem. We remark that it is usually not necessary to condition $T_1$ exactly
on $W$---it is typically enough to condition on a
larger $\sigma$-algebra. However, some form of conditioning is usually necessary
(typically averaging over all possible `small parts' that can be perturbed).

We note at this point that there is an infinitesimal
version of the perturbation idea, introduced by \cite{Stein1995} and further
elaborated by \cite{Meckes2008}. It is applicable if the underlying random
variables are continuous. Starting with an identity of the form \eq{6} with
non-trivial $T_1(\eps)$ and $T_2(\eps)$ (in the original work by means of
classic exchangeable pairs) depending on some $\eps>0$, a preceding limit
argument $\eps\tozero$ yields an identity \eq{6} with non-trivial $T_1$ but~$T_2
= 0$.

Having a coupling $(W,W',G)$ at hand one is already in a position to apply our
theorems and obtain bounds of closeness to normality in the respective metric.
However, bounding $\Var\IE^W T_1 = \Var\IE^W(GD)$ is not always easy and
sometimes not optimal as one typically has to make use of (truncated) fourth
moments of the involved random variables. For this reason it is often
beneficial---and sometimes crucial---to introduce other auxiliary random
variables. 

The first extension is to replace conditioning on $W$ by conditioning on
another random variable $W''$, which is still assumed to be close to $W$ but
typically independent of $GD$ (but not independent of $W$ and
$W'$!). In this case the main error term becomes $\Var^{W''}(GD)$
which of course vanishes if $W''$ is independent of~$GD$. Although this comes
at the cost of additional error terms, these are usually easier to bound. We
need to emphasize that the distribution of $W''$ is---a priori---irrelevant and
no corresponding equation of the form \eq{8} has to be satisfied for $W''$; the
only important feature is independence from $GD$ and
closeness to~$W$. In fact, we can show that, if $(W,W',G)$ is a Stein coupling
and $W''$
is independent of $GD$, then we can construct $T_3$ and $T_4$ such that
\be
  \IE\bklg{Wf(W)} = \IE f'(W) + \IE \bklg{T_3 f''(W+T_4)}
\ee
for all smooth enough functions~$f$. Compared to \eq{6}, this is clearly a step
further towards~\eq{1}. Typically, the specific dependence structure in $W$
used to construct $W''$ independently of $GD$ can also be exploited to
calculate bounds on $\Var\IE^W(GD)$---so why introducing $W''$ in the first
place? The crucial advantage is that the dependence structure can be
exploited in a more direct way (that is, in an earlier stage of the proof),
avoiding forth moments---and constructing $W''$ is often easier
than bounding~$\Var\IE^W(GD)$.

Other improvements can be made in specific applications by replacing $D$ by a
random variable $\~D$ such that $\IE^W(G\~D) = \IE^W(GD)$ (note that we use the
same letter only for convenience; $\~D$ itself does not have to be the
difference of two random variables) and replacing $1$ in \eq{7} by a random
variable $S$ such that $\IE^W S = 1$. Smart choices of $\~D$ and $S$ may allow
us
to construct $W''$ to be closer to $W$ in order to improve or simplify the error
bounds. As we will see in Section~\ref{sec7} about decomposable random
variables, the use of $\~D$ can be crucial. However,
at first reading one may always assume that $W'' = W$, $\~D = D$ and~$S=1$.

We do not claim that \emph{all} results that have been obtained using Stein's
method for normal approximation can be represented in terms of these auxiliary
random variables; but we provide evidence that we can cover a large part
of them. It may be possible to setup an even more general framework by
introducing other auxiliary random variables so that even very specialised
results such as \cite{Sunklodas2008}, where very fine and delicate calculations
are necessary to obtain optimal rates, could be represented in such a framework.
We did not attempt to do so in order to keep our theorems manageable. 

As mentioned after \eq{6}, zero-biasing couplings, as introduced in
\cite{Goldstein1997}, are the special case of \eq{6} where~$T_1=1$. It
is thus not surprising, that we are not able to directly represent a
zero-bias coupling as a Stein coupling. However, we will show that each
coupling $(W,W',G)$ satisfying \eq{8} gives rise to a zero-bias construction
(but, unfortunately, the construction does not directly lead to a coupling
with~$W$). Based on an exchangeable pair, \cite{Goldstein2005a} propose a way to
construct the zero bias~$W^z$. We adapt this construction to our more general
setting. As such, the zero-bias approach \emph{parallels} our approach rather
than being
a special case of it. Therefore, in this article, we take a different point of
view than, for example, \cite{Goldstein1997} and \cite{Goldstein2005}, where
size and zero biasing are seen to be closely related (both are distributional
transformations). From our perspective, size biasing is closer related to
approaches such as local approach and exchangeable pairs approach and we think
of zero biasing as being separate from these.

The rest of the article is organized as follows. In the remainder of the
introduction we will introduce the metrics of interest. In Section~2 we will
present the main theorems of the article and discuss the crucial error terms. In
Section~3 we will show how known approaches fit into our framework and also
present and discuss new couplings. Section~4 is dedicated to some
applications in order to see different couplings in action.
In Section~5 we will make the connection with the zero-bias approach and in
Section~6 we will prove the main results from Section 2.

\subsection{The probability metrics}
For probability distribution functions $P$ and $Q$ define 
\be
    \dw(P,Q) = \int_{-\infty}^\infty \abs{P(x)-Q(x)}dx,
    \qquad
    \dk(P,Q) = \sup_{x\in\IR}\abs{P(x)-Q(x)}.
\ee
The first quantity is known as $L_1$,  \emph{Wasserstein} or \emph{Kantorovich}
metric and is only a metric on the set of probability distributions with finite
first moment. If $X\sim P$ and $Y\sim Q$
have finite first moments and if $\%F_{\mathrm{W}}$ is the set of Lipschitz
continuous functions on $\IR$ with Lipschitz constant at most $1$, we have 
\be
    \dw(P,Q) = \sup_{h\in\%F_{\mathrm{W}}}\babs{\IE h(X) - \IE h(Y)},
\ee
where the infimum ranges over all possible couplings of $X$ and~$Y$. The second
metric is known as \emph{Kolmogorov} or \emph{uniform metric} and if
$\%F_{\mathrm{K}}$ denotes the set of half line indicators we obviously have
\be
    \dk(P,Q) = \sup_{h\in\%F_{\mathrm{K}}}\babs{\IE h(X) - \IE h(Y)}.
\ee
If $\phi$ is a right continuous function on $\IR$ such that for each $\eps$ we
have
\be
    Q(B^{\eps})\leq Q(B)+\phi(\eps)
\ee
for all Borel set $B\subset\IR$, where $B^\eps=\{x\,:\,\inf_{y\in
B}\abs{x-y}\leq \eps\}$, then
\be
    \dk(P,Q) \leq \sqrt{\dw(P,Q)+\phi(\dw(P,Q))}
\ee
(see \cite{Gibbs2002} for a compilation of this and similar results). If
$Q=\N(0,1)$ is the standard normal measure we can take $\phi(x)=x\sqrt{2/\pi}$,
thus
\ben                                                        \label{14}
    \dk(P,\N(0,1)) \leq 1.35\sqrt{\dw(P,\N(0,1))}.
\ee
However, as is well known, this bound is often not optimal and, in fact, in many
situations both metrics will exhibit the same rates of convergence.

\section{Main results} \label{sec2}

Throughout this section, let all random variables be at least square integrable
and defined on the same probability space without making any further assumptions
unless explicitly stated. Again, we point out that at first reading one
may set $W'' = W$, $\~D = D$ and~$S=1$. It may also be helpful to keep in
mind the introductory couplings for sums of independent random variables
in Section~\ref{sec1}. Before stating our main theorems,
we first define and discuss some error terms in order to express the overall
error bounds in the different metrics and using different techniques. Let
\ben                                                        \label{15}
    r_0 = \sup_{\norm{f},\norm{f'}\leq 1}{\babs{\IE\bklg{Gf(W')-Gf(W)-
	Wf(W)}}},
\ee
where the supremum is meant to be taken over all function $f$ which are bounded
by $1$ and Lipschitz continuous with Lipschitz constant at most~$1$. We need to
point out that in the proofs the actual supremum is only taken
over the solutions to the Stein equation so that in cases where more
properties of $f$ are needed (such as better constants) this can easily be
accomplished. Clearly, if
$(W,W',G)$ is a Stein coupling, then~$r_0=0$. Hence, for the couplings and
examples discussed in this article the actual set of functions over which the
supremum is taken is not relevant. In cases where the
linearity condition is not exactly satisfied (see e.g.\ \cite{Rinott1997} and
\cite{Shao2005}) $r_0$ measures the corresponding error; \cite{Shao2005}
handles self-normalised sums where uniformly bounded derivatives
of $f$ are
needed to proof that (the implicitly used) $r_0$ is small. Let now
\be
    D:=W'-W,\qquad D':=W''-W,
\ee
let $\~D$ be a square integrable and let $S$ be an 
integrable random variable on the same probability space. Define
\be 
    r_1 = \IE\babs{\IE^W(GD-G\~D)},
    \quad
    r_2 = \IE\babs{\IE^W(1-S)},
    \quad
    r_3 = \IE\babs{\IE^{W''}(G\~D-S)}.
\ee
Clearly, $r_1$ is the error we make by replacing $D$ by $\~D$, $r_2$ the error
we make by replacing $1$ by $S$ and $r_3$ corresponds to the main error term as
discussed in the introduction. These error terms will appear irrespective of
the metric. Additional error terms which are specific to the metric of interest
will be defined in the respective sections. 

\subsection{Wasserstein distance}

Bounds for smooth test functions are typically easier to obtain as no
smoothness of $W$ is required. Let us first start with a very general theorem,
from which we will then deduce some simpler corollaries.

\begin{theorem} \label{thm1}
Let $W$, $W'$, $W''$, $G$ and $\~D$ be square integrable random variables and
let $S$ be an integrable random variable. Then 
\bes
   &\dw\bklr{\law(W),\N(0,1)} \\
   &\qquad \leq 2r_0 + 0.8 r_1+ 0.8r_2 + 0.8 r_3 + 1.6 r_4 + r_5 +  1.6 r_4'
         + 2 r_5',
\ee
where
\ba
    r_4 &= \IE\abs{GD\,\I[\abs{D}>1]},
    &r_4' &=\IE\abs{(G\~D-S)\I[\abs{D'}>1]}, \\
    r_5 &= \IE\abs{G(D^2\wedge 1)},
    &r_5' & = \IE\abs{(G\~D-S)(\abs{D'}\wedge 1)}.\\
\ee
\end{theorem}

Note that, here and in later results, the truncation constant~$1$ is not chosen
arbitrarily as the truncation at $1$ has some optimality properties; see
\cite{Loh1975} and \cite{Chen2001}. The difference $\abs{G\~D-S}$ in $r_4'$ and
$r_5'$ can usually be replaced by $\abs{G\~D}+\abs{S}$ without much loss of
precision.

Under additional, but not too strong assumptions we have the following.
\begin{corollary}\label{cor1}
Let $(W,W',G)$ be a Stein coupling with $\Var W = 1$. Then, under
finite fourth moments assumption,
\ben                                                            \label{16}
  \dw\bklr{\law(W),\N(0,1)}\leq 0.8\sqrt{\Var\IE^W(GD)}+\IE\abs{GD^2}.
\ee
\end{corollary}
\begin{proof} Theorem~\ref{thm1} with $S=1$, $\~D=D$ and $W''=W$
yields the result, but with bigger constants. However, with the
assumption of finite fourth moments no truncation is necessary, and one can
obtain the better result~\eq{16} with essentially the same proof as for
Theorem~\ref{thm1}. 
\end{proof}

\begin{corollary}\label{cor2}
Let $(W,W',G)$ be a Stein coupling with $\Var W = 1$ and assume that there
are $S$ and $\~D$ such that $\IE^W S = 1$, $\IE^W(G\~D)=\IE^W (GD)$ and $W''$
independent of $(G\~D,S)$.
Then, under finite third moments of $W$, $W'$, $G$ and~$\~D$ and
$\IE\abs{S}^{3/2}<\infty$,
\be
  \dw\bklr{\law(W),\N(0,1)}\leq \IE\abs{GD^2} + 2\IE\abs{G\~DD'} +
2\IE\abs{SD'}.
\ee
\end{corollary}

Using H\"older's inequality, we obtain the following straightforward
simplification which gives the correct order in `typical' situations; we
assume~$S=1$.

\begin{corollary}\label{cor3} Let $(W,W',G)$ be a Stein coupling with $\Var W =
1$ and assume that there is
$\~D$ such that $\IE^W(G\~D)=\IE^W (GD)$ and $W''$
independent of~$G\~D$.
Then, if
\ben                                                           \label{17}
  \IE\abs{D}^3\vee \IE\abs{\~D}^3\vee \IE\abs{D'}^3 \leq A^3,\qquad
  \IE\abs{G}^3 \leq B^3,
\ee
for some positive constants $A$ and $B$, we have
\ben                                                            \label{18}
  \dw\bklr{\law(W),\N(0,1)}\leq 5A^2B.
\ee
\end{corollary}
\begin{proof} The only part which may need explanation is that $\IE(GD) =1$
because $\Var W = 1$ by assumption and hence  $1\leq \IE\abs{GD}$. This yields
\be
  \IE\abs{D'}\leq \IE\abs{GD}\IE\abs{D'}.
\ee
The remaining part is due to H\"older's inequality.
\end{proof}

\subsection{Kolmogorov distance}

Let us now consider bounds with respect to the Kolmogorov metric. To
obtain such results, all approaches using Stein's method will
estimate probabilities of the form 
\ben                                                    \label{19}
    \IP[A\leq W \leq B \,\mid\, \%F]
\ee
at some stage of the proof, where $A$ and $B$ are $\%F$-measurable random
variables. This is in order to control the smoothness of $W$ as we are now
dealing with a non-smooth metric. Essentially two techniques can be found in the
literature to deal with such expressions. In the first approach, bounds on
\eq{19} are established by using similar techniques as used for Stein's method,
i.e.\ Identity~\eq{12} is applied for special functions $f$ to obtain an
explicit bound on \eq{19}. We call this the \emph{concentration inequality
approach}; see \cite{Ho1978}, \cite{Chen2005}, \cite{Shao2006a},
\cite{Chatterjee2006}. In the second approach, \eq{19} is bounded in an
indirect way in terms of $\dk\bklr{\law(W|\%F),\N(\mu_{\%F},\sigma_{\%F}^2)}$
where $\mu_{\%F}$ and $\sigma_{\%F}^2$ are the conditional expectation and
variance, respectively, of~$W$ (see Lemma~\ref{lem8} in Section~\ref{sec14})
which will lead to some form of recursive inequality; hence we will call this
the \emph{recursive approach}. Such inequalities are either solved directly,
i.e.\ if $\%F$ is the trivial $\sigma$-algebra (see \cite{Rinott1997} and
\cite{Raic2003}), or by using an inductive argument, making use of some
additional structure in $W$; this is typically the case if $\%F$ is a
non-trivial $\sigma$-algebra (see \cite{Bolthausen1984}). The recursive approach
has the advantage that an explicit bound of \eq{19} is not needed. This comes at
cost of more structure in the coupling. 

As has been observed by many authors, such as \cite{Rinott1997}, \cite{Raic2003}
or \cite{Chen2005} in the context of Stein's method, it is easier to obtain
Kolmogorov bounds under some
boundedness conditions, in which case the recursive approach can be easily
implemented, in fact, with $\%F$ equal to the trivial $\sigma$-algebra. Using a
truncation argument, boundedness can be relaxed, but in order to obtain useful
results one will need fast decaying tails of $G$, $D$, $\~D$ and~$D'$. We will
use mainly this approach for our applications; see also for example
\cite{Shao2006a} and \cite{Chatterjee2006}.

\begin{theorem} \label{thm2}
Let $W$, $W'$, $G$, $\~D$, $W''$ be square integrable random
variables and $S$ be an integrable random variable. Then, for any non-negative
constants $\alpha$, $\beta$, $\beta'$, $\~\beta$ and~$\gamma$
\bes
   &\dk\bklr{\law(W),\N(0,1)}\\
   &\enskip\leq 2\bklrl r_0+r_1 + r_2 + r_3 + r_6 + r_6'   
+(\alpha\~\beta+\gamma)(\IE\abs{W}+5)\beta'+(\IE\abs{W}+3)\alpha\beta^2\bklrr. 
\ee 
where
\ba    
r_6& = \IE\babs{GD\,
        \I[\text{$\abs{G}>\alpha$ or $\abs{D}>\beta$}]},\\
    r_6'& = \IE\babs{(G\~D-S)
        \I[\text{$\abs{G}>\alpha$ or $\abs{\~D}>\~\beta$ or
        $\abs{D'}>\beta'$ or $\abs{S}>\gamma$}]}.
\ee
\end{theorem}

If a sequence of couplings $(W_n,W'_n,G_n,W_n'',\~D_n,S_n)_{n\geq 1}$ is under
consideration, the truncation points $\alpha$, $\beta$, $\beta'$, $\~\beta$ and
$\gamma$ will of course
need to dependent on~$n$. In a typical situation, say a sum of $n$ bounded
i.i.d.\ random variables, we will have $\alpha\asymp n^{1/2}$,
$\beta\asymp\beta'\asymp\~\beta\asymp n^{-1/2}$ and $\gamma\asymp 1$. 

\begin{corollary}\label{cor4} Let $(W,W',G)$ be a Stein coupling with $\Var W =
1$. If $G$
and $D$ are bounded by positive constants $\alpha$ and $\beta$, respectively,
then
\be
  \dk\bklr{\law(W),\N(0,1)} \leq 2\sqrt{\Var \IE^W(GD)} + 8\alpha\beta^2
\ee
\end{corollary}

Note that, even if $G$ and $D$ are bounded, we unfortunately cannot deduce
a direct, useful bound on $\Var \IE^W(GD)$ from that fact. Instead, we need
again more structure in order to avoid $\Var \IE^W(GD)$.

\begin{corollary}\label{cor5} Let $(W,W',G)$ be a Stein coupling with $\Var W =
1$ and assume that there are $\~D$ such that $\IE^W(G\~D)=\IE^W (GD)$, $S$ such
that $\IE^W S = 1$ and $W''$ independent of~$(G\~D,S)$.
If the absolute values of $G$, $D$, $\~D$, $D'$ and $S$ are bounded by $\alpha$,
$\beta$, $\~\beta$, $\beta'$ and $\gamma$, respectively, then
\be
  \dk\bklr{\law(W),\N(0,1)} \leq 8\alpha\beta^2 + 12\alpha\~\beta\beta' 
    +12\gamma\beta'.
\ee 
\end{corollary}

We need to emphasize the remarkable statement of Corollary~\ref{cor5}: under the
conditions stated we immediately obtain a bound on the Kolmogorov distance to
the standard normal without any additional computations! Examples can easily
found such as bounded, locally dependent random variables; see
Section~\ref{sec5}.

The approach we will use for the next theorem was developed by \cite{Chen2004a}
for locally dependent random variables. Although a concentration
inequality approach was used by \cite{Chen2004a}, the
recursive approach is easy to implement without loss of precision. Like in
Theorem~\ref{thm2}, the
aim is to obtain a bound involving \eq{19} with respect to the
unconditional~$W$.
This comes at the cost of truncated forth moments, especially in the form
of~$r_8$. Hence, the approach of avoiding truncated forth moments by making use
of $W''$ in $r_3$ will not be useful because of the presence of~$r_8$.
Therefore, we give below only a version for $W'' = W$, $\~D = D$ and $S=1$ to
avoid
unnecessary overloading of the bound. To define some additional error terms,
let
\bgn                                                            
   \^K(t) := G (\I[0 \leq t < D] - \I[D\leq t< 0]),             \label{20}\\
   K^W(t):=\IE^W\^K(t),\qquad K(t) := \IE\^K(t)\notag.
\ee
%Alternatively, we can write $\^K(t) = G\sgn(D)\I[D\notin(-t,t)]$.

\begin{theorem} \label{thm3}
Let $W$, $W'$ and $G$ be square integrable random variables on the same
probability space. Then 
\bes
   \dk\bklr{\law(W),\N(0,1)}
    &\leq 2r_0 + 2\^r_3 + 2r_4 + 2(\IE\abs{W}+2.4)r_5 \\
    &\quad + 1.4r_7 + 2((\IE(\abs{W}+1)^2)^{1/2}+1.1)r_8
\ee
where $\^r_3 = \IE\babs{\IE^W(GD)-1}$, where $r_4$
and
$r_5$ are defined as in Theorem~\ref{thm1} and where
\be
    r_7 = \int_{\abs{t}\leq 1 }\Var K^W(t) dt,
    \qquad r_8 = \bbbklr{\int_{\abs{t}\leq 1 }\abs{t}\Var K^W(t) dt}^{1/2}.
\ee

\end{theorem}

Let us discuss $\Var K^W(t)$. Typically, $W$ will consist of $n$
parts, such as a sum of $n$ random variables or a functional in $n$ coordinates.
In this case, the perturbation $W'$ will typically be constructed by picking a
small part of $W$, say part $I$, where $I$ is uniform on $\{1,\dots,n\}$ and
then perturb this part. Thus, we typically will have $G:=nY_I$, $D := D_I$ for
sequences $Y_1,\dots,Y_n$ and $D_1,\dots,D_n$ and then define $W':= W+D_I$,
where, with $\sigma^2 = \Var W$, $\IE\abs{Y_i}=\bigo(\sigma^{-1/2})$ and
$\IE\abs{D_i}=\bigo(\sigma^{-1/2})$.
Let now $(W^*,W^{*\prime},G^*)$ be an independent copy of $(W,W',G)$ and
let $\I^\pm_t(x) = \I[0\leq t<x] - \I[x\leq t < 0]$. Then we can write
\bes
    \Var K^W(t) &\leq 
    \sum_{i,j=1}^n \Cov(Y_i\I^\pm_t(D_i),Y_j\I^\pm_t(D_j))\\    
    & = \sum_{i,j=1}^n \IE\bklg{Y_iY_j\I^\pm_t(D_i)\I^\pm_t(D_j)
    -Y_iY^*_j\I^\pm_t(D_i)\I^\pm_t(D^*_j)}
\ee
so that 
\bes
    r_7 & \leq \sum_{i,j=1}^n
        \IE\bklg{Y_iY_j\I[D_iD_j>0](\abs{D_i}\wedge\abs{D_j}\wedge1)\\   
    &\qquad\qquad\qquad
    -Y_iY^*_j\I[D_iD^*_j>0](\abs{D_i}\wedge\abs{D^*_j}\wedge1)}
\ee
and
\bes
    r_8^2 & \leq \frac{1}{2}\sum_{i,j=1}^n
        \IE\bklg{Y_iY_j\I[D_iD_j>0](\abs{D_i}^2\wedge\abs{D_j}^2\wedge1)\\   
    &\qquad\qquad\qquad
    -Y_iY^*_j\I[D_iD^*_j>0](\abs{D_i}^2\wedge\abs{D^*_j}^2\wedge1)},
\ee
respectively. In the case of local dependence, these quantities can now be
bounded relatively easily; see \cite{Chen2004a}.

If truncated fourth moments are to be avoided and no boundedness can
be assumed, it seems that more structure is needed in the coupling. A typical
instance is the
use of higher-order neighbourhoods under local dependence as in
\cite{Chen2004a}, or the recursive structure in the combinatorial CLT in
\cite{Bolthausen1984}. The theorem below is the basis for such results and it
contains expressions of the form \eq{19} explicitly, so that further steps are
needed for a final bound. 

Define for a random element $X$ defined on the same probability space as~$W$   
the quantity 
\be
    \vartheta_\eps(X) = \sup_{a\in\IR}\IP[a\leq W \leq a+\eps\,|\,X],
\ee
where we assume without further mentioning that the regular conditional
probability exists. 

\begin{theorem} \label{thm4}
Let $W$, $W'$, $W''$, $\~D$ and $G$ be random variables with finite third
moments and $S$ be a random variable with $\IE\abs{S}^{3/2}<\infty$. Then, for
any $\eps>0$,
\besn                                                          \label{21}
    &\dk\bklr{\law(W),\N(0,1)} \\
    &\quad\leq r_0 + r_1 + r_2 + r_3 + r_9 + 0.5 r_{10}  + \eps^{-1}
r_{11}(\eps) + 0.5\eps^{-1}r_{12}(\eps)+ 0.4\eps
\ee
where
\ba  
r_9 &=\IE\babs{(S-G\~D)(\abs{W}+1)(\abs{D'}\wedge 1)}
&r_{10} &=\IE\abs{G(\abs{W}+1)(D^2\wedge 1)}\\
r_{11}(\eps) &=\IE\babs{(S-G\~D)D' \vartheta_\eps(G,\~D,D',S)}
&r_{12}(\eps) &=\IE\babs{GD^2 \vartheta_\eps(G,D)}
\ee
\end{theorem}

We now look at a method to obtain a final bound from the above theorem using
induction. Although never mentioned in the literature around Stein's method,
this type of argument can be traced back to \cite{Bergstrom1944}, who
uses Lindeberg's method and an inductive argument to prove a Kolmogorov bound in
the CLT. The argument was used later by \cite{Bolthausen1982a} in the context of
martingale central limit theorems. The following Lemma~\ref{lem1} provides the
key element to the
inductive approach in the
context of Stein's method as introduced by \cite{Bolthausen1984}. It can be used
to obtain a final bound from an estimate of the form \eq{21}, provided that $W =
W_n$ has some recursive structure and $S_\eps$ can be expressed in terms of the
closeness of $W_1,W_2,\dots,W_{n-1}$ to the standard normal distribution. Note
that in the following lemma, the numbers $\kappa_k$, $k=1,\dots,n$ denote the
respective bounds on the Kolmogorov distance between $W_k$ and the standard
normal. Whereas \cite{Bolthausen1984} uses a recursion involving
$\kappa_{n-4},\dots,\kappa_{n-1}$, \cite{Goldstein2010} introduces a version
involving all possible $\kappa_1,\dots,\kappa_{n-1}$ to prove Berry-Esseen type
bounds for degree counts in the \ER\ random graph using size biasing.
Incidentally, already \cite{Bergstrom1944} uses $\kappa_1,\dots,\kappa_{n-1}$
for his inductive argument, although his argument is of a somewhat different
flavour. The following lemma is inspired by the work of \cite{Goldstein2010},
but
adapted to be used along with Theorem~\ref{thm4}. An independent proof will be
given in Section~\ref{sec14}.

\begin{lemma}\label{lem1} Let $\kappa_1,\dots,\kappa_n$, be a sequence of
non-negative numbers such that $\kappa_1 \leq 1$. Assume that there is a
constant $A\geq 0$, a triangular array $A_{k,1},\dots,A_{k,k}\geq 0$,
$k=2,3,\dots,n$, and a sequence $\sigma_2,\dots,\sigma_n> 0$ such that, for
all $\eps>0$ and all $2\leq k\leq n$,
\ben                                                            \label{22}
    \kappa_k \leq \frac{A}{\sigma_k} + 0.4\eps +
    \frac{1}{\eps\sigma_k}
    \sum_{l=1}^{k-1} A_{k,l} \kappa_l.
\ee
Then,
\be
%     \kappa_n\leq \frac{(A+0.4c\alpha_n)c}{\sigma_n(1-c)}
\kappa_n\leq \frac{1}{\sigma_n}\frac{\bklr{5(A\vee 1)+2\alpha_n+
\alpha_n'
}\bklr{2\alpha_n+\alpha_n'}
}{5\alpha_n'},
\ee
where
\be
    \alpha_n = \sup_{2\leq k\leq
n}\sum_{l=1}^{k-1}\frac{\sigma_k }{\sigma_l}A_{k,l},
\qquad \alpha_n' = \sqrt{2\alpha_n(2\alpha_n+5(A\vee1))}.
\ee
\end{lemma}

\begin{example} Let $W_n=n^{-1/2}\sum_{i=1}^n X_i$ where $X_i$ are i.i.d.\ with
$\IE X_i=0$ and $\Var X_i = 1$ and $\IE\abs{X_i}^3 =  \gamma \geq 1$. Let
$\kappa_n = \dk\bklr{\law(W_n),\N(0,1)}$. Set
$G = -n^{1/2}X_I$ and $W' = W - n^{-1/2}X_I$. Set also $W'' = W'$, $\~D = D$
and $S=1$. Hence $D = D' = -n^{-1/2}X_I$. We have
\be
    r_0 = r_1 = r_2 = r_3 = 0.
\ee
Furthermore,
\be
    r_9 \leq 6\gamma/\sqrt{n}, \qquad r_{10} \leq 3\gamma/\sqrt{n}
\ee
Note now (c.f. Lemma~\ref{lem8})
\bes
    \vartheta_\eps(X_n) 
    & = \sup_a\IP\kle{a\leq W_n \leq a + \eps \mid  X_n}\\
    & = \sup_a\IP\bbkle{\sqrt{\tsfrac{n}{n-1}}a\leq 
        W_{n-1} + \sqrt{\tsfrac{1}{n-1}}X_n \leq
       \sqrt{\tsfrac{n}{n-1}}a + \sqrt{\tsfrac{n}{n-1}}\eps + \bbmid X_n}\\
    & \leq \sqrt{\tsfrac{n}{2\pi(n-1)}}\eps + 2\kappa_{n-1} 
    \leq \eps + 2\kappa_{n-1},
\ee
hence
\be
    r_{11}(\eps) \leq 2\gamma(\eps + 2\kappa_{n-1})/\sqrt{n},\qquad
    r_{12}(\eps) \leq \gamma(\eps + 2\kappa_{n-1})/\sqrt{n}.
\ee
Putting these estimates into Theorem~\ref{thm4} we obtain
\be
    \kappa_k \leq \frac{8\gamma}{\sqrt{k}} + 0.4\eps +
\frac{5\gamma}{\eps\sqrt{k}}\kappa_{k-1}.
\ee
We can apply Lemma~\ref{lem1} with $A = 8\gamma$, $A_{k,k-1} = 5\gamma$
and $A_{k,l} = 0$ for $l < k-1$, and $\sigma_k = k^{-1/2}$. We have
$\alpha_n = 5\gamma\sqrt{2}$, thus, 
plugging this into~\eq{22}, $\kappa_n \leq 25\gamma/\sqrt{n}$. As this example
illustrates, the constants obtained this way are typically not optimal,
but nevertheless explicit.
\end{example}

\section{Couplings} \label{sec3}

In this section we present some well-known and some new couplings and show how
they can be represented in our general framework. The basis is always the
coupling $(W,W',G)$ and throughout this section (with the exception of classic
exchangeable pairs) we will only look at cases of actual Stein couplings, that
is, where~$r_0=0$. This implies in particular that $\IE W = 0$ (which,
nevertheless, has to be assumed explicitly in some cases to make the
construction work in the first place). Unless otherwise stated, the
variance~$\sigma^2$ of
$W$ is arbitrary, but finite and non-zero. Note that, if $(W,W',G)$ is a Stein
coupling, so is $(W/\sigma,W'/\sigma,G/\sigma)$, and hence we will usually omit
the standardising constant $\sigma^{-1}$ for ease of notation. To simplify or
optimize the
bounds, we sometimes will extend the coupling by different choices of $\~D$, 
$W''$ and~$S$. But, if not otherwise stated, we will make the basic assumption
throughout this section that $\~D = D$, $W''=W$ and~$S=1$.

We mostly present the construction of the couplings only and
not the particular form of the final bounds for the normal approximation. The
reason for this is that, once the coupling is constructed, one can directly
apply our theorems or corollaries of the main section to obtain the
corresponding bounds. Hence, stating them explicitly would be either just
repeating known results from the literature or rephrasing the results from the
main section. 

We need to clarify again that a Stein coupling by itself does by no
means imply closeness to normality or imply any convergence. As can be
seen from the case of quadratic forms (Section~\ref{sec8}), Stein couplings as
defined by \eq{8} can also be used for $\chi^2$ approximation.

Let throughout this part $[n] := \{1,2,\dots,n\}$ and $[0] := \emptyset$. Let
also in general $I$ and $J$ be independent random variables, uniformly
distributed on $[n]$ and independent of all else, but we will usually mention
this---and deviations from it---explicitly whenever we make use of these random
variables.

\subsection{Exchangeable pairs and extensions}\label{sec4}
This approach was introduced by Stein in a paper by \cite{Diaconis1977}. A
systematic
exposition was given by \cite{Stein1986}. 

\begin{constr}\label{con1} Assume that\/ $(W,W')$ is an exchangeable pair. If,
for some constant $\lambda>0$, we
have
\ben                                                             \label{23}
    \IE^W(W'-W) = -\lambda W,
\ee
then $(W,W',\frac{1}{2\lambda}(W'-W))$ is a Stein coupling.
\end{constr}

\cite{Rinott1997} generalised
\eq{23} to allow for some non-linearity in~\eq{23}; however, the resulting
coupling will only be an approximate Stein coupling.
 
\begin{subconstr} Assume that\/ $(W,W')$ is an exchangeable pair where $\IE
W = 0$ and\/ $\Var W = 1$. Assume that, for some constant $\lambda>0$, we
have
\ben                                                             \label{24}
    \IE^W(W'-W) = -\lambda W + R.
\ee
then, with $G=\frac{1}{2\lambda}(W'-W))$,
\be
    r_0 \leq \lambda^{-1}\IE\abs{R}, \qquad
    \abs{\IE(GD)-1} \leq \lambda^{-1}\abs{\IE(WR)} \leq \lambda^{-1}\sqrt{\Var
      R}.
\ee
(note that we use $\Var W =1$ only to obtain the last two
inequalities).
\end{subconstr}

The conditional expectation can of course always be written in the form of
\eq{24} for any~$\lambda$. However, we will need
$\lambda^{-1}\sqrt{\Var{R}}\tozero$
to obtain convergent bounds, and in this sense the choice of $\lambda$ is, at
least asymptotically, unique; see the discussion in the introduction of
\cite{Reinert2009}.

Note that we call this approach `classic' for this specific choice of~$G$.
There are many other ways to construct Stein couplings where $(W,W')$ is
exchangeable but $G$ is not a multiple of $W'-W$; we will give such examples
later on.

The classic exchangeable pairs approach is frequently used in the literature;
see for example \cite{Rinott1997}, \cite{Fulman2004}, \cite{Fulman2004a},
\cite{Rollin2007a}, \cite{Meckes2008} and others. Generally, one constructs a
``natural'' exchangeable pair $(W',W)$ and then hopes that \eq{24} holds with
$R=0$ or $R$ small enough to yield convergence. However, more often than not,
this will not succeed, even for simple examples as the $2$-runs examples below
illustrates. Based on work by \cite{Reinert2009}, we will present in
Section~\ref{sec4b} Stein couplings making use of a multivariate extensions of
\eq{23}
which will lead to appropriate modifications of $G$ such that \eq{8} holds. 
In Sections~\ref{sec9} and~\ref{sec10} we will also present two very
general couplings that are based on exchangeable pairs, but where $G$ is
chosen rather differently.

A few more detailed remarks about this approach are appropriate here. For this
specific choice of $G=(W'-W)/2\lambda$, \cite{Rollin2008} proves that
exchangeability is actually not necessary to prove a result such as
Theorem~\ref{thm2}, as long as we have equal marginals 
$\law(W')=\law(W)$. \cite{Rollin2008} uses a different way of deducing
a Stein identity of the form~\eq{12}. With $F(w)=\int_0^w f(x)dx$ one obtains
from Taylor's expansion that
\be
    F(W') - F(W) = 
    Df(W) + D\int_0^D (1-s/D) f'(W+s)ds,
\ee
so that, again with $G=D/2\lambda$, and assuming
$\law(W')=\law(W)$,
\ben                                                            \label{25}
    \IE\{G f(W)\} = \IE \bbbklg{G\int_0^D (1-s/D)f'(W+D)ds},
\ee
which serves as a replacement for~\eq{12}. 
In contrast, \cite{Stein1986} uses the antisymmetric function approach. If
$(W,W')$ is exchangeable then $\IE \bklg{(W'-W)(f(W')+f(W))} = 0$, and it is
not difficult to show that
\ben                                                            \label{26}
    \IE\klg{Gf(W)}  = \frac{1}{2}\IE \bbbklg{G\int_0^D f'(W+D)ds}.
\ee
Note that this is almost \eq{25} except that the factor $(1-s/D)$ is replaced
by~$1/2$. Note again that \eq{26} is only true under exchangeability whereas
\eq{25} holds for equal marginals. Surprisingly, better constants can be
obtained if \eq{25} is used instead of \eq{26}, although exchangeability is a
stronger assumption. Incidentally, in Section~\ref{sec13}, a coupling is used
which is not an exchangeable pair but has equal marginals, however, in the
context of a different construction than Construction~\ref{con1}.

% \subsubsection{Ad-hoc corrections to $G$} 
% 
% We propose to choose a different $G$, instead, to
% obtain linearity. The idea is
% to modify $G$ in such a way that the non-linear part from the right hand side
% of
% \eq{27} is `annihilated'. 
% 
% \begin{subconstr} $W$ and $W'$ as above and
% \be
%     G =
%     \tsfrac{n}{2\sigma_S}\bklr{\xi'_I\xi_{I+1}-\xi_I\xi_{I+1}+p\xi_I'-p\xi_I}
% \ee
% define a Stein coupling.
% \end{subconstr}
% 
% This is because
% \ba
%     \IE{Gf(W)} 
%     &=\frac{1}{2\sigma_V}
%      
% \IE\sum_{i=1}^n\bklr{\xi'_i\xi_{i+1}-\xi_i\xi_{i+1}+p\xi_i'-p\xi_i}f(W)\\
%     &=\frac{1}{2\sigma_V}
%     \IE\sum_{i=1}^n\bklr{p\xi_{i+1}-\xi_i\xi_{i+1}+p^2-p\xi_i}f(W)\\
%     &=-\frac{1}{2\sigma_V}
%     \IE\sum_{i=1}^n\bklr{\xi_i\xi_{i+1}-p^2}f(W) =
% -\frac{1}{2}\IE\bklg{Wf(W)}.
% \ee
% Noting that $\xi_i$ is independent of $W'_i$ we can prove in the same way
% that $\IE\bklg{Gf(W')}=\ahalf\IE\bklg{Wf(W)}$, so that \eq{8} is satisfied.

\subsubsection{Multivariate exchangeable pairs}\label{sec4b}

In \cite{Reinert2009}, the classic exchangeable pairs approach was generalised
to $d$-dimensional vectors $W=(W_1,\dots,W_d)$ and $W'=(W'_1,\dots,W'_d)$ which
satisfy
\ben							\label{27b}
  \IE^W(W'-W) = -\Lambda W
\ee
for some invertible $(d\times d)$-matrix~$\Lambda$. They are able
to obtain multivariate normal approximation results in cases where the
exchangeable pair of univariate random variables $(W_1,W_1')$ does not satisfy
\eq{23}, but, using auxiliary random variates, an
embedding of that pair into a higher dimensional space satisfies \eq{27b}.
However, the transition to higher dimensions comes at the cost of having to
impose stronger conditions on the set of test functions. Hence, besides the
multivariate
approximation, it is therefore still of
interest to examine $W_1$ directly. It turns out that, once the higher
dimensional embedding satisfying \eq{27b} is found, it is easy to
construct a Stein coupling from that. 

\begin{subconstr}\label{con1b} Let $(W,W')$ be an exchangeable pair of
$d$-dimensional random vectors satisfying \eq{27} for some invertible~$\Lambda$.
Let $e_i$ be the $i$-th unit vector. Then
\be
  \bklr{W_i,W_i',\ahalf e_i^t\Lambda^{-1}(W'-W)}
\ee
is a Stein coupling.
\end{subconstr}

Indeed, 
\bes
  -\IE\bklg{Gf(W_i)}
  & = -\ahalf\IE\bklg{ e_i^t\Lambda^{-1}\IE^W(W'-W)f(W_i)}\\
  & = \ahalf\IE\bklg{ e_i^t\Lambda^{-1}\Lambda W f(W_i)}
    = \ahalf\IE\bklg{W_i f(W_i)}.
\ee  
and, using exchangeability, the corresponding result for~$\IE\bklg{
Gf(W'_i)}$ can be obtained in the same way. Hence, every multidimensional
exchangeable pair $(W,W')$ satisfying \eq{27b} gives rise to a univariate
Stein coupling for each individual coordinate.

Let us consider the case of $2$-runs on a circle. To this
end, let $\xi_1,\dots,\xi_n$ be a sequence of independent $\Be(p)$ distributed
random
variables. Let $V = \sum_{i=1}^n (\xi_i\xi_{i+1}-p^2)$ be the centered number of
$2$-runs,
where we put $\xi_{n+1} = \xi_1$ (hence `circle'). Consider now the
following coupling. With $\xi'_1,\dots,\xi'_n$ being independent copies of
$\xi_1,\dots,\xi_n$, let
$V' = V -\xi_{I-1}\xi_{I}-\xi_{I}\xi_{
I+1 }+ \xi_{I-1}\xi'_{I}+\xi'_{I}\xi_{I+1}$, where $I$ is
uniformly distributed on $[n]$ and independent of all else. It is easy to
see that $(V,V')$ is an exchangeable pair and that 
\ben                                                        \label{27}  
    \IE^V(V'-V) = -\frac{2}{n}V + \frac{2p}{n}\sum_{i=1}^n (\xi_i-p).
\ee
Even in this very simple example, the linearity condition \eq{24} cannot
be obtained with the above natural coupling. Based on the same exchangeable
pair, \cite{Reinert2009} use the embedding method to circumvent this problem. To
this end we introduce the auxiliary statistic $U = \sum_{i=1}^n(\xi_i-p)$ and
define $U' = U - \xi_I + \xi_I'$. Condition \eq{27b} is now
satisfied for $W=(U,V)$, $W'=(U',V')$ and 
\be
    \Lambda = \frac{1}{n}\left[\begin{array}{cc}1&0\\-2p&2\end{array}\right].
\ee
The inverse of $\Lambda$ is 
\be
  \Lambda^{-1} = n\left[\begin{array}{cc}1&0\\p&1/2\end{array}\right],
\ee
hence Construction~\ref{con1b} yields
\besn							\label{27d}
  G & = \frac{n}{2}\bklr{p(U'-U)+\ahalf(V'-V)}\\
    & = \frac{n}{2}\bklr{p\xi'_I -
p\xi_I+\ahalf\xi_{I-1}\xi'_{I}+\ahalf\xi'_{I}\xi_{I+1}-\ahalf\xi_{I-1}\xi_{I}
-\ahalf\xi_{I} \xi_{I+1}}
\ee
so that $(V,V',G)$ is a Stein coupling. Exploiting some specific properties in
this example, we may also choose
\ben							\label{27c}
  G  = \frac{n}{2}\bklr{p\xi'_I-p\xi_I +\xi'_I\xi_{I+1}-\xi_I\xi_{I+1}}
\ee
to obtain a somewhat simpler Stein coupling, but the similarity between
\eq{27d} and \eq{27c} is apparent. See
\cite{Reinert2009}, \cite{Reinert2009a} and \cite{Ghosh2010a}
for further examples of multivariate exchangeable pair couplings. 

\subsubsection{Finding $W$ for a given coupling and a given $G$} In some cases
it may not be clear from the beginning how to choose the main random variable
of interest~$W$.
Consider the Curie-Weiss model of ferromagnetic interaction. With $\beta\geq 0$
being the \emph{inverse temperature} and $h\in\IR$ the \emph{external field}
on
the state
space $\{-1,1\}^n$, we define the probabilities for each
$\sigma=(\sigma_1,\dots,\sigma_n)\in\{-1,1\}^n$ by the Gibbs measure  
\ben                                                    \label{28}
    \IP[\{\sigma\}] = Z^{-1}
    \exp\bbbklg{\frac{\beta}{n}\sum_{i<j}\sigma_i\sigma_j + h\sum_i\sigma_i}
\ee
where $Z=Z(\beta,h,n)$ is the partition function to make the probabilities sum
up to~$1$. A quantity of interest is the magnetization $m(\sigma) = n^{-1}\sum_i
\sigma_i\in[-1,1]$ of the system. However, in the low-temperature regime the
system will exhibit \emph{spontaneous magnetization} so that we may not be
interested in $m(\sigma)$ itself if $\sigma$ is drawn at random according
to~\eq{28} but in $m(\sigma)$ \emph{relative} to its corresponding
magnetization. To find a suitable correction term (which
shall serve here as an illustrative example only),
\cite{Chatterjee2007} proposes the following construction.

\begin{subconstr} Let $(\sigma,\sigma')$ be an exchangeable pair on some
measure space and let $\phi(\sigma,\sigma')$ be an anti-symmetric
function. Let $G = -\phi(\sigma,\sigma')/2$, $W = W(\sigma) =\IE^\sigma
\phi(\sigma,\sigma')$ and $W' = W(\sigma') =
\IE^{\sigma'}\phi(\sigma',\sigma)$. Then this defines a Stein coupling.
\end{subconstr}

Indeed,
\bes
  -\IE\bklg{Gf(W)} = \ahalf\IE\bklg{\phi(\sigma,\sigma')f(W)}
  = \ahalf\IE\bklg{Wf(W)}
\ee
and
\bes
  \IE\bklg{Gf(W')} & = -\ahalf\IE\bklg{\phi(\sigma,\sigma')f(W(\sigma'))}\\
   & = -\ahalf\IE\bklg{\phi(\sigma',\sigma)f(W(\sigma))}\\
   & = \ahalf\IE\bklg{\phi(\sigma,\sigma')f(W)} 
    =  \ahalf\IE\bklg{Wf(W)},
\ee
which proves \eq{8}.

Let us apply this to the Curie-Weiss model. First, given $\sigma$ is drawn from
\eq{28}, we define $\sigma'$ by choosing a
site $I$ uniformly at random and then we re-sample this site according to the
conditional distribution $\law(\sigma_I \,\mid\, \sigma_j,j\neq I)$, giving a
new $\sigma_I'$, but leaving all the other sites untouched. Now we set $\phi(
\sigma,\sigma')=
n(m(\sigma) - m(\sigma')) = \sigma_I-\sigma'_I$. It is not
difficult to show that
\be
    \IE^\sigma\phi(\sigma,\sigma')=
m(\sigma)-\frac{1}{n}\sum_{i}\tanh(\beta
m_i(\sigma)+\beta h),
\ee
where $m_i(\sigma) = \frac{1}{n}\sum_{j\neq i}\sigma_j$. Hence, we let 
$G = -(\sigma_I-\sigma'_I)/2$, 
\be
    W = W(\sigma) =
        m(\sigma)-\frac{1}{n}\sum_{i}\tanh(\beta m_i(\sigma)+\beta h),
\ee
and $W' := W(\sigma')$. As 
\be
    \bbabs{\tanh(\beta m(\sigma) + \beta h)
    - \frac{1}{n}\sum_{i}\tanh(\beta m_i(\sigma)+\beta h)} \leq
\frac{\beta}{n},
\ee  
we may alternatively choose $W = m(\sigma) - \tanh(\beta m(\sigma)+\beta h)$, in
which
case \eq{8} is not satisfied anymore, but we still have $r_0 \leq
\beta/n$.
The key here is to find $\phi(\sigma,\sigma')$ such that $\IE^\sigma
\phi(\sigma,\sigma')$ yields `something interesting'. In \cite{Chatterjee2007}
this construction is used to prove concentration of measure results for
such~$W$.

\subsubsection{Finding an antisymmetric $G$ through a Poisson
equation}

Assume now that $(X,X')$ is an exchangeable pair on some space
$\%X$ and $W = \phi(X)$ and $W' = \phi(X')$ for some functional
$\phi:\%X\to\IR$ with $\IE\phi(X)=0$.
\citet[Section~4.1]{Chatterjee2005a} proposes a general approach to find $G$
of a special form. Let $G(X',X) = \ahalf(\psi(X')-\psi(X))$
for some unknown functional $\psi:\%X\to\IR$ (in fact, \emph{any} anti-symmetric
function can be written in this form; see \cite{Stein1986}). Using
exchangeability, it is not
difficult to see that \eq{8} is satisfied if 
\be
    \psi(x)- P \psi (x) = \phi(x)
\ee
for every $x\in\%X$, which we recognize as a Poisson equation with
kernel $P\psi(x):=\IE^{X=x} \psi(X')$ for given $\phi$ and unknown~$\psi$.
A general (formal) solution is $\psi(x) = \sum_{k=0}^\infty P^k\phi(x)$. We
have the following.

\begin{subconstr}\label{con2}
Let $(X,X')$ be an exchangeable pair on a measure space $\%X$ and let
$\phi:\%X\to\IR$ be a measurable function such that $\IE\phi(X) = 0$; define
$P$ as above. If there is
a constant $C>0$ such that
\ben                                  \label{29}
  \sum_{k=0}^\infty \abs{P^k \phi(x) - P^k \phi(y)} \leq C
\ee
for every $x,y\in\%X$, then 
\be
  (W,W',G)=\bbklr{\phi(X),\phi(X'), \frac{1}{2}\sum_{k=0}^\infty \bklr{P^k
\phi(X) - P^k \phi(X')}}
\ee 
is a Stein coupling. 
\end{subconstr}

Note that boundedness $\abs{G}\leq C/2 $ is built in through
\eq{29} so that this construction is  a natural candidate for
Theorem~\ref{thm2}. We can give a more constructive version of this coupling.

\begin{subconstr} Assume that $(X,X')$, $\phi$ and $P$ are as in
Construction~\ref{con2}. Assume that we have two
Markov chains $\bklr{X_{n}}_{n\geq 0}$ and $\bklr{X'_{n}}_{n\geq 0}$ with
the transition dynamics given by $P$, and also $X_{0} = X$ and $X'_{0} = X'$.
Assume further that, for all $n$,
\ben
  \law\bklr{X_{n}\big| X,X'} = \law\bklr{X_{n}\big| X},\quad
  \law\bklr{X'_{n}\big| X,X'} = \law\bklr{X'_{n}\big| X'}.
\ee
Let now $T=\inf\bklg{n > 0\bmid X_{n} = X'_{n}}$ be the coupling
time of the two chains and assume that\/ $T < \infty$ almost surely. If, given
$T$, $I$ is uniformly distributed on $\{0,1,2,\dots,T-1\}$, then 
\ben                                                            \label{30}
  (W,W',G)=\bklr{\phi(X),\phi(X'), \ahalf T \bklr{\phi\klr{X_{I}} -
\phi\klr{X'_{I}} }}
\ee
is a Stein coupling. 
\end{subconstr}

Indeed, from
\be
  \IE\bklg{\phi(X_k)f(W)} = \IE \bklg{P^k(X)f(W)}
\ee
and
\be
  \IE\bklg{\phi(X'_k)f(W)} = \IE\bklg{f(W)\IE^{X,X'}\phi(X'_k)}
  = \IE\bklg{f(W)P^k\phi(X')},
\ee
we easily obtain
\bes
  \IE\bklg{T \bklr{\phi\klr{X_{I}} - \phi\klr{X'_{I}} }f(W)}
  & = \IE\sum_{k=0}^{T-1}\bklr{\phi\klr{X_{k}} - \phi\klr{X'_{k}} }f(W)\\
  & = \sum_{k=0}^{\infty}\IE\bklg{\bklr{\phi\klr{X_{k}} -
\phi\klr{X'_{k}}}f(W)}\\
  & = \sum_{k=0}^{\infty}\IE\bklg{\bklr{P^k\phi\klr{X} -
P^k\phi\klr{X'}}f(W)},
\ee
and, similarly, 
\bes
  \IE\bklg{T \bklr{\phi\klr{X_{I}} - \phi\klr{X'_{I}} }f(W')}
  &= \sum_{k=0}^{\infty}\IE\bklg{\bklr{P^k\phi\klr{X} -
P^k\phi\klr{X'}}f(W')}\\
  &= - \sum_{k=0}^{\infty}\IE\bklg{\bklr{P^k\phi\klr{X} -
P^k\phi\klr{X'}}f(W)},
\ee
where the second step uses exchangeability. Hence, it follows from
Construction~\ref{con2} that \eq{30} is a Stein coupling; see
\citet[Section~4.1]{Chatterjee2005a} for
more details, and see also \cite{Makowski1994} on general theory about Poisson
equations.

\subsection{Local dependence and related couplings} \label{sec5}

This is one of the earliest versions of Stein's method. Let in what follows $I$
be uniformly distributed on $[n]$, independent of all else. 

\begin{constr}\label{con3}Let $W=\sum_{i=1}^n X_i$ with~$\IE X_i=0$.
For each $i$, let $W'_i$ be such that
\ben                                                    \label{31}
  \IE\bklr{X_i\,|\, W'_i }=0.
\ee 
Then, $(W,W',G)=(W,W'_I,-nX_I)$ is a Stein coupling. 
\end{constr}

To see this we have on one hand
\ben                                             
      \label{32}
    -\IE\bklg{Gf(W)} = \sum_{i=1}^n \IE\bklg{X_if(W)} = \IE\bklg{Wf(W)},
\ee
and on the other hand 
\be 
    \IE\bklg{Gf(W')} = -\sum_{i=1}^n \IE\bklg{X_i f(W'_i)} = 0,    
\ee
due to~\eq{31}; hence \eq{8} is satisfied. 

The choice $G=-nX_I$ was first considered by \cite{Stein1972} for $m$-dependent
sequences, however this $G$ has broader applications. We now discuss some more
detailed constructions of $W'_i$ below.

\subsubsection{Local dependence} \label{sec6}

Local dependence was extensively studied in \cite{Chen2004a} under various
dependence settings, but of course this approach goes back to
\cite{Stein1972} and \cite{Chen1975}; a version for discrete random variables
is given by \cite{Rollin2008a}.  We can use the simplest form as a
starting point
 
\begin{subconstr}\label{con4} Assume that $W$ and $G$ are as in
Construction~\ref{con3}. Assume in addition that, for each $i\in[n]$,
there is $A_i\subset [n]$ such that $X_i$ and $(X_j)_{j\in A_i^c}$ are
independent. Then, with $W'_i = W - \sum_{j\in A_i} X_j$, \eq{31} is
satisfied. 
\end{subconstr}

This first-order dependence is usually referred to as (LD1) and is
enough to obtain a Stein coupling. However, it is possible to extend this
coupling.

\begin{subconstr}\label{con5} Assume that $W$ and $G$ and $W_i$ are as in
Constructions~\ref{con3} and~\ref{con4} and that~$\Var W = 1$. Assume in
addition that there is $B_i\subset [n]$ such that $A_i\subset B_i$ and
$(X_j)_{j\in A_i}$ and $(X_j)_{j\in B_i^c}$ are independent. Define $W_i'' = W -
\sum_{j\in B_i} X_j$; then $W'' := W_I''$ is independent of $GD$ and hence~$r_3
= 0$.
\end{subconstr}

The conditions of Constructions~\ref{con4} and~\ref{con5} together are
referred to as (LD2).

\subsubsection{Decomposable random variables} \label{sec7}
This version of local dependence was popularized by
\cite{Barbour1989} for smooth test functions, by \cite{Raic2004} for
Kolmogorov distance and by \cite{Rollin2008a} for total variation
approximation of discrete random variables. It uses a refined version of the
concept of second-order neighborhood and makes use of non-trivial~$\~D$ and~$S$.

\begin{subconstr} Assume that $W$ and $G$ and $W_i$ are as in
Constructions~\ref{con3} and~\ref{con4}. Assume in addition that, for each $i$
and for each $j\in A_i$, there is $B_{i,j}\subset[n]$ such that $A_i\subset
B_{i,j}$ and such that $(X_i,X_j)$ is independent of
$(X_j)_{j\in B^c_{i,j}}$. Let $K_i = \abs{A_i}$ and define $\~D = K_I X_J$
where, given $I$, $J$ is uniformly distributed on $[K_I]$, but independent of
all else. Then, $\IE^X(G\~D) = \IE^X(GD)$, hence~$r_1 = 0$. Let $S =
nK_I\sigma_{I,J}$ where $\sigma_{i,j}=\IE(X_iX_j)$, then~$r_2=0$. Define
$W_{i,j}'' =
W - \sum_{k\in B_{i,j}} X_k$; then $W'' :=
W_{I,J}''$ is independent of $G\~D$ and $S$ and hence~$r_3 = 0$.
\end{subconstr}

Hence, if we can choose $B_{i,j}$ such that $B_{i,j}\subsetneq B_i$, where $B_i$
is as for the standard (LD2) local dependence setting from
Construction~\ref{con5}, we should be able to
improve our bounds, as $W''-W = \sum_{k\in B_{I,J}} X_k$ contains fewer
summands as compared to $\sum_{k\in B_{I}} X_k$ from Construction~\ref{con5}.

Note that, under third moment conditions, \cite{Barbour1989} obtain a
Wasserstein bound of order
\ben                            \label{33}            
  \sum_{i=1}^n \IE \abs{X_i Z_i^2}
  +\sum_{i=1}^n\sum_{j\in A_i}\bklr{\IE \abs{X_i
      X_j V_{i,j}}  + \abs{\IE(X_i X_j)}\IE\abs{Z_i+V_{i,j}}}
\ee
with
\be
  Z_i := \sum_{j\in A_i} X_k, 
  \qquad V_{i,j}:=\sum_{k\in B_{i,j}\setminus A_i} X_k,
\ee
(note that in \cite{Barbour1989} the coarser expression $\IE\abs{X_i X_j}$ is
used, but it is easy to see that this can be sharpened to $\abs{\IE (X_i
X_j)}$. 
Using Corollary~\ref{cor2}, we obtain an order of
\ben                            \label{34}
  \sum_{i=1}^n \IE \abs{X_i Z_i^2} 
   +\sum_{i=1}^n\sum_{j\in A_i}\bklr{\IE\abs{X_i X_j (Z_i+V_{i,j})}
   +\abs{\IE(X_i X_j)}\IE\abs{Z_i+V_{i,j}}}.
\ee
In most cases we can expect that these two bounds will yield similar results:
Indeed, a useful upper bound on both estimates is 
\be
  \sum_{i=1}^n\sum_{j\in A_i}\sum_{k\in B_{i,j}} 
    \bklr{\IE\abs{X_iX_jX_k} +\abs{ \IE(X_iX_j)}\IE\abs{X_k}} 
\ee
up to constants. 

Consider the von Mises statistics as an example, where, for independent random
variables $X_1,\dots,X_n$, we have $W=\sum_{p,q}\phi_{p,q}(X_p,X_q)$ for some
functionals~$\phi_{p,q}$. Clearly, we can choose $A_{(p,q)} =
\{(k,l)\,:\,\text{$k=q$ or $l=q$}\}$ as first-order neighborhood. However, in
the standard local approach framework, we would need to let $B_{(p,q)} =
[n]\times[n]$ for (LD2) so that $W'' = 0$ and hence $\abs{D'}=\abs{W}$ which
would not yield useful bounds. In the refined setting we can choose, for every
$(p',q')\in A_{(p,q)}$, the set $B_{(p,q),(p',q')} :=
\{(k,l)\,:\,\text{$k\in\{p,p'\}$ or $l\in\{q,q'\}$}\}$, so that now
$\abs{B_{(p,q),(p',q')}}$ is only of order~$n$. Of
course, it will depend on the concrete choice of functionals $\phi_{p,q}$
whether normal approximation is appropriate at all; see \cite{Barbour1989} for
applications to random graph related statistics.

\subsubsection{Special case: quadratic forms}\label{sec8} Let
$\xi_1,\dots,\xi_n$ be
independent, centered random variables with unit variance. Let
$A=(a_{ij})$ be a real
symmetric $(n\times n)$-matrix. Let $W = \bklr{\sum_{i,j}
a_{ij}\xi_i\xi_j - \sum_i a_{ii}}$. It would be straightforward to use the above
method of decomposable random variables in this situation. However, due to the
multiplicative structure (or, in $U$-statistics language, because the kernel
$\phi_{i,j}(x,y)=a_{ij}xy$ is degenerate for centered random variables) there is
an interesting alternative.

\begin{subconstr} Let $W$ be as above. Let also $Y_i := \sum_j a_{ij}\xi_j$, $G
= -n(\xi_I
Y_I - a_{II})$ and $W' = W - (2\xi_I Y_I - a_{II}\xi_I^2)$. Then
this defines a Stein coupling. 
\end{subconstr}

It is not
difficult to see that \eq{8} holds. Again, it depends on the matrix $A$
whether we can expect normal like behaviour of $W$ or not: essentially this
coupling was used by \cite{Luk1994} in the context of $\chi^2$-approximation
for the case where all the entries of $A$ are $1$, corresponding to the
square of a sum of random variables.

\subsubsection{Local exchangeable randomization}\label{sec9}
This coupling was proposed in \cite{Reinert1998a}. Its use was limited by the
fact that, if the classic exchangeable pairs
approach (as discussed in Subsection~\ref{sec4}) is used along with this
coupling, the linearity condition \eq{24} will in general not be satisfied with
$R$ small enough, but with the choice $G=-nX_I$, we can now handle
this
coupling. However, some care is needed. 

\begin{subconstr}\label{con6} Let $W$ and $G$ be as in
Construction~\ref{con3}.
Let $(X'_{i,j})_{i,j\in[n]}$ be a
collection of random variables,
such that, with $W'_i = \sum_{j=1}^n X'_{i,j}$, we have that
\ban
  (i)&\quad\text{for each $i$, $X'_{i,i}$ is independent of $W$,}\label{37}\\
%   (ii)&\quad\text{for each $i$, $\law\klr{W,X'_{i,i}} = \law\klr{W'_i,X_i}$.} 
  (ii)&\quad\text{for each $i$, $\bklr{(X_k)_k,(X'_{i,k})_k}$ is exchangeable.}
                                                       \label{38}
\ee                                                   
Then \eq{31} is satisfied.
\end{subconstr}

It is often not too difficult to construct $(X'_{i,j})_j$ for a given $i$ such
that $\law(W'_i) = \law(W)$ and such that $X'_{i,i}$ is independent of $W$
and hence $\IE^W X'_{i,i} = 0$. However, it is important to note that this does
not suffice as we ultimately need $\IE^{W'_i} X_{i} = 0$, which is, however,
guaranteed under the additional Condition~\eq{38}.

In \cite{Reinert1998a}, it was incorrectly deduced from \eq{37} and the property
 \ben                                                   
\label{39}
    \law(X'_{i,j},j\neq i\,|\,X'_{i,i}=x)=\law(X_j,j\neq i\,|\,X_i=x)
\ee
that $(W,W'_i)$ is exchangeable. It is not difficult to find examples for which
$(W,W'_i)$ is not exchangeable, but \eq{37} and \eq{39} are still
true; see Remark~\ref{rem2}.

\subsection{Size-biasing}

This approach was introduced in \cite{Baldi1989} and further explored in
\cite{Goldstein1996}, \cite{Dembo1996}, \cite{Goldstein2008} and others.

\begin{constr} Let $V$ be a non-negative random variable with $\IE V = \mu>0$.
Let $V^s$ have the \emph{size-biased distribution}
of~$V$, that is, for all bounded~$f$,
\ben                                                    \label{40}
    \IE\bklg{Vf(V)} = \mu\IE f(V^s).
\ee  
Then
\be
  (W,W',G) = \bklr{V-\mu, V^s-\mu, \mu}
\ee
is a Stein coupling.
\end{constr}

Using \eq{40}, we obtain
\bes
    \IE\bklg{Gf(W') - Gf(W)}
    & = \IE\bklg{\mu f(V^s-\mu)-\mu f(V-\mu)}\\
    & = \IE\bklg{V f(V-\mu)-\mu f(V-\mu)}\\
    & =\IE\bklg{Wf(W)},
\ee
so that \eq{8} is satisfied, indeed. 

One of the advantages of this approach is apparent if bounds for the Kolmogorov
metric are to be obtained. In the light of Theorem~\ref{thm1}, we see that
$G/\sigma$ is already bounded by $\alpha=\mu/\sigma$, so that we only need to
concentrate on
finding a bounded coupling $(W,W')$; see \cite{Goldstein2008} for such a
coupling in the context of coverage problems.

\subsection{Interpolation to independence}\label{sec10} 

For this coupling the key idea is to construct a sequence of random
variables that `interpolates' between $W$ and an independent copy of $W$ by
means of small perturbations. A
special case of this coupling was introduced by \cite{Chatterjee2008}. The
construction has apparent similarities to Lindeberg's telescoping
sum in his prove of the CLT for sums of independent random variables.
Let in the following construction $I$ be uniformly
distributed on $[n]$ and independent of all else. 

\begin{constr}\label{con7} Assume $\IE W = 0$. Assume that
for
each $i\in[n]$ we have a $W'_i$ which is close to~$W$. Assume that there is
a sequence of random variables $V_0,V_1,\dots,V_n$ such that $\IE^W V_0
=W$ and such that $V_n$ is independent of $V_0$
and assume that, for every $i\in[n]$, 
\ben                                                                \label{42}
    \law\bkle{(W,V_{i-1}),(W'_i,V_{i})} = \law\bkle{(W'_i,V_{i}),(W,V_{i-1})}
\ee
for every~$i\in[n]$. Then
\ben                                                                \label{43}
  (W,W',G) = \bklr{W,W'_I,\tsfrac{n}{2}(V_{I}-V_{I-1})}  
\ee
is a Stein coupling.
\end{constr}

Note that \eq{42} implies in particular that $(W,W'_i)$ is an exchangeable pair
for each $i$ and also that $\law(V_i)=\law(V_0)$ for all $i$ by induction. We
have
\bes 
    \IE\bklg{Gf(W)} & = \ahalf\IE\sum_{i=1}^{n}(V_i-V_{i-1})f(W) \\
    & = \ahalf\IE\bklg{(V_n-V_0)f(W)}\\
    & =-\ahalf\IE\bklg{Wf(W)},
\ee
due to the independence assumption, and, due to \eq{42},
\bes
    \IE\bklg{G f(W')}
    &=\frac{1}{2}\sum_{i=1}^n\IE\bklg{\klr{V_i-V_{i-1}}f(W'_i)}\\
    &=\frac{1}{2}\sum_{i=1}^n\IE\bklg{\klr{V_{i-1}-V_i}f(W)}\\
    &=-\frac{1}{2}\sum_{i=1}^n\IE\bklg{\klr{V_i-V_{i-1}}f(W)}\\
    &=-\IE\bklg{Gf(W)} = \ahalf\IE\bklg{Wf(W)}.
\ee 
Hence, \eq{43} is a Stein coupling, indeed.

\subsubsection{Functionals of independent random variables} A specific version
of this coupling was used by \cite{Chatterjee2008} for functionals of
independent random variables. We give a simpler version first and discuss then
the (implicitly used) coupling of \cite{Chatterjee2008}.

\begin{subconstr}\label{con8}Let $X=(X_1,\dots,X_n)$ be a collection of
independent random variables and let $W=F(X)$ be any functional of $X$ such that
$\IE F(X) = 0$. Let $X'=(X'_1,\dots,X'_n)$
be an independent copy of $X$ and define for all subsets $A\subset [n]$ the
vectors $X^A = (X^A_1,\dots,X^A_n)$ by 
\be
    X^A_i = \begin{cases}
             X'_i & \text{if $i\in A$,}\\
             X_i  & \text{if $i\notin A$,}
            \end{cases}
\ee
that is, $X^A$ is simply $X$ but with all $X_i$ replaced by $X'_i$ for which
$i\in A$; define also $W'_A = F(X^A)$. Let $W'_i := W'_{\{i\}}$ for each
$i\in[n]$ and $V_i := W'_{[i]}$ for each $i\in
[n]\cup\{0\}$. Then the conditions of Construction~\ref{con7} are satisfied.
\end{subconstr}

Clearly, $V_n$ is independent of $V_0$ and it is not difficult to see that
\eq{42} holds. The interpolating sequence is therefore constructed simply by
replacing the $X_i$ by $X'_i$ in increasing order. For this coupling to be
useful we would typically need that $F$ is not too sensitive to changes in the
individual coordinates.

The implicit coupling used by \cite{Chatterjee2008} is different in the sense
that, instead of using a fixed order in which the $X_i$ are replaced, a random
order is used. 

\begin{subconstr}\label{con9}Assume that $W$, $F$, $X$ and $X'$ are as in
\ref{con8}. Let $\Pi$ be a uniformly drawn random permutation of
length $n$, independent of everything else. For any permutation $\pi$ we denote
by $\pi(A)$ simply the image of $A$ with respect to~$\pi$. Define now $W'_i :=
W'_{\{\Pi(i)\}}$ and $V_i := W'_{\Pi([i])}$. Then the conditions of
Construction~\ref{con7} are satisfied.  
\end{subconstr} 

Exchangeability \eq{42} follows
from Construction~\ref{con8} by conditioning on~$\Pi$. Let us
now prove that Construction~\ref{con9} indeed leads to the representation used
by
\cite{Chatterjee2008}. Clearly, $G =
\frac{1}{2n}(W'_{\Pi([I])}-W'_{\Pi([I-1])})$, thus
\bes
    \IE^{X,X'}(GD) 
    &=\frac{1}{2n!}\sum_{i=1}^n\sum_{\pi}
        (W'_{\pi([i])}-W'_{\pi([i-1])})(W-W'_{\{\pi(i)\}}).
\ee
We re-write the sum over all permutation as a sum over all possible subsets
induced by $\pi([i-1])$, that is, all possible subsets $A\subset[n]$ with
$\abs{A} = i-1$, and over all possible values of $\pi(i)$ which range over
$[n]\setminus A$. Taking into account multiplicities from 
all possible permutations within the sets $\pi([i-1])$ and
$\pi([n]\setminus[i])$ we obtain
\bes
    &\IE^{X,X'}(GD)\\
      &\quad=\frac{1}{2n!}\sum_{i=1}^n\sum_{A\subset[n],\atop\abs{A}=i-1}
         \sum_{j\notin
          A}\abs{A}!(n-\abs{A}-1)!(W'_{A\cup\{j\}}-W'_A)(W-W'_{\{j\}})\\
     &\quad=\frac{1}{2}\sum_{A\subsetneq[n]}
         \sum_{j\notin A}
   \frac{1}{{n\choose\abs{A}}(n-\abs{A})}(W'_{A\cup\{j\}}-W'_A)(W-W'_{\{j\}}),
\ee
which is exactly the expression used by \cite[Eq.~(1)]{Chatterjee2008}.

\subsection{Local symmetry}

An instance of this coupling was used by \cite{Chen1998} for sums of independent
random variables.

\begin{constr}\label{con10}
Assume that $W$, $W'$, $G_\alpha $ and $G_\beta $ are random variables such that
\ban                             
    \IE\klg{G_\alpha  f(W')} &= \IE\klg{G_\beta  f(W')},    \label{45}\\
    \IE\klg{G_\alpha  f(W)} &= \IE\klg{W f(W)},              \label{46}\\
    \IE\klg{G_\beta f(W)} &= 0,                               \label{47}
\ee
for all $f$ for which the expectations exist. Then, $(W,W',G_\beta -G_\alpha
)$ is a Stein coupling. 
\end{constr}

Indeed, using \eq{45} for the first equality and then
\eq{46} and \eq{47}, 
\be
    \IE\klg{(G_\beta -G_\alpha )(f(W')-f(W))} = \IE\klg{(G_\alpha-G_\beta) 
f(W)} = \IE\klg{Wf(W)}.
\ee
Note that we refer to Condition~\eq{45} as \emph{local symmetry} due to the
following example.

Let $X=(X_1,\dots,X_n)$ be a sequence of centered independent
random variables and let $X'$ be an independent copy of~$X$. Let $W=\sum_i X_i$
and assume that~$\Var W = 1$. Define $G_\alpha  = X_I$ and $W' = W + X'_I$ (we
`duplicate' a small part of $W$). Define also $G_\beta  = X_I'$. Then it is not
difficult to verify Conditions \eq{45} and~\eq{46}. Identity~\eq{45}
is due to the symmetry of $G_\alpha $ and $G_\beta $ relative
to $W'$, which is the crucial aspect of the construction. 

\subsection{Abstract approaches}

One might wonder if, for a given arbitrary coupling $(W,W')$, one can always
find $G$ to make $(W,W',G)$ a Stein coupling. 

\begin{constr} Let $(W,W')$ be a pair of integrable random variables. Let
$\%F$ and $\%F'$ be two $\sigma$-algebras with $\sigma(W)\subset \%F$
and $\sigma(W')\subset \%F'$. Let $V$ be a random variable such that
\ben						\label{47b}
  \IE^W V = W.
\ee
Define (formally) the random variable
\be
  G = -V + \IE(V|\%F') - \IE(\IE(V|\%F')|\%F)
      + \IE(\IE(\IE(V|\%F')|\%F)|\%F') - \dots.
\ee
If the above sequence converges absolutely almost surely, then $(W,W',G)$ is a
Stein coupling. 
\end{constr}

To see this, first condition $G$ on $\%F'$ which yields $\IE(G|\%F') = 0$. On
the other hand, $\IE(G|\%F) = -\IE(V|\%F)$ hence $\IE^W G = -W$ and \eq{8}
follows. Note that \eq{47b} is not to be confused with the usual linearity
condition \eq{23}---we can always take $V = W$ to satisfy \eq{47b}.

Consider the example $W = \sum_{i=1}^n X_i$, a sum of independent, centered
random variables. With $I$ independent and uniformly distributed on $[n]$, let
$W' = W - X_I$. Take $V = W$ and note that
\be
  \IE(W|W') = \IE(W'+X_I|W') = W',\qquad  \IE(W'|W) = (1-\anth)W.
\ee
Hence,
\bes
  -G & = W - W' + (1-\anth)W - (1-\anth)W'+ (1-\anth)^2W - (1-\anth)^2W'+\dots\\
    & = X_I + (1-\anth)X_I +(1-\anth)^2 X_I +  \dots = n X_I.
\ee
Alternatively, chosing $V = nX_I$ yields the same result directly, as
$\IE^{W'}V = 0$ and hence $G=-V = -nX_I$.

\section{Applications}

In this section we give some applications of the main theorems and
corollaries using different couplings from
Section~\ref{sec3}.
Table~\ref{tab} gives an overview over the different couplings  we use in
this section along with some important characteristics of the
involved random variables.

\begin{table}
\footnotesize
\begin{tabular}{|c|c|c|c|c|}
\cline{2,3,4,5}
\multicolumn{1}{c|}{}&\vbox to4ex{} Construction & $W'\eqlaw W$ &
\raisebox{0.9ex}{\vtop to 5ex{\hsize=2cm $(W',W)$ exchangeable}} & $W''\eqlaw
W$\\
\hlx{hv}
Hoeffding (Var.~1)&\ref{con1} (p.~\pageref{con1}) &yes & yes & --\\
\hlx{v}
Hoeffding (Var.~2)&\ref{con6} (p.~\pageref{con6})&yes & yes & --\\
\hlx{v}
Hoeffding (Var.~3)&\ref{con10} (p.~\pageref{con10})&no & no & yes\\
\hlx{v}
Occupancy &\ref{con6} (p.~\pageref{con6}) &yes & yes & yes\\
\hlx{v}
Neighbourhood &\ref{con3} (p.~\pageref{con3})&no & no & no\\
\hlx{v}
Random graphs &\ref{con3} (p.~\pageref{con3})&yes & no & yes\\
\hlx{h}
\end{tabular}
\medskip
\caption{\label{tab} Overview over the couplings used in the
different applications along with some interesting properties of the involved
random variables.}
\end{table}

\subsection{Hoeffding's combinatorial statistic} \label{sec12}

Let $a_{i,j}$, $1\leq i,j \leq n$, be real numbers such that
$\sum_{k=1}^n a_{i,k} = \sum_{k=1}^ n a_{k,j} = 0$ and $\frac{1}{n-1}\sum_{i,j}
a_{i,j}^2=1$. Let $\pi$ be a uniformly chosen random permutation of size $n$ and
$W = \sum_{i=1}^n a_{i,\pi(i)}$. Then it is routine to see that $\IE W = 0$
and~$\Var W = 1$. Note that, for a Stein coupling $(W,W',G)$, unit variance
of $W$ implies~$\IE(GD)=1$.  Let in what follows $I_1$ and $I_2$ be
independent and uniformly chosen random numbers from $[n]$. 

\medskip

\paragraph{\bf Variant 1 (c.f.\ Construction~\ref{con1})} Define $\pi' = \pi
\circ (I_1\, I_2)$ and $W' = \sum_{i=1}^n a_{i,\pi'(i)}$. Then $(W,W')$ is a
classical exchangeable pair, i.e. \eq{23} holds with $\lambda = 2/n$, or,
equivalently, $G = \tsfrac{n}{4}(W'-W) =
\tsfrac{n}{4}(a_{I_1,\pi(I_2)}+a_{I_2,\pi(I_1)}-a_{I_1,\pi(I_1)}-a_{I_2,\pi(I_2)
} )$ makes $(W,W',G)$ a Stein coupling.

\medskip

\paragraph{\bf Variant 2 (c.f.\ Construction~\ref{con6})} Define $W'$ as in
Variant~1. With $G = -n a_{I_1,\pi(I_1)}$, $(W,W',G)$ is also a Stein
coupling. This coupling is the (implicit) basis for the construction in
\cite{Ho1978} and \cite{Bolthausen1984}.

\medskip

In both of the previous variants, our $W'$ is defined with respect to a
perturbation $\pi'$ of~$\pi$. Thus, $W'$ can be seen as an instance of the
\emph{replacement perturbation} from the introduction. This comes at the cost of
$D$ having four terms. One may wonder whether a \emph{deletion} construction 
is possible, where $W'$ is defined by just `removing a random small part' of $W$
(see Section~\ref{sec6}). This is possible, indeed, so that we do not need to go
through constructing~$\pi'$. Despite the fact that the following construction is
very simple, it has gone unnoticed in the
literature so far.

\medskip 

\paragraph{\bf Variant 3 (c.f.\ Construction~\ref{con10})} Define 
\be
  W' = W - 
  \begin{cases}
    (a_{I_1,\pi(I_1)} + a_{I_2,\pi(I_2)}) & \text{if $I_1\neq I_2$,}\\
    a_{I_1,\pi(I_1)} &\text{if $I_1 = I_2$.}
  \end{cases}
\ee
Let $G = n(a_{I_1,\pi(I_2)}-a_{I_1,\pi(I_1)})$; then
$(W,W',G)$ is a Stein coupling. To see this, note first that $\sigma(W')
\subset \%F := \sigma\bklr{I_1,I_2,(\pi(i);i\neq I_1,I_2)}$. Now, if $I_1\neq
I_2$, the conditional distributions $\law(\pi(I_1)|\%F)$ and
$\law(\pi(I_2)|\%F)$ are equal
and assign probability $1/2$ to each of the points in the set 
$\{\pi(I_1),\pi(I_2)\}$ so that $\IE\klg{Gf(W')} = 0$. The same arguments
from Variant 1 lead to $\IE\klg{Gf(W)} = -\IE\klg{Wf(W)}$. 
 
To see the connection with Construction~\ref{con10}, let $G_\alpha = n
a_{I_1,\pi(I_1)}$ and $G_\beta = n a_{I_1,\pi(I_2)}$; then \eq{45}--\eq{47} are
satisfied.

\medskip

Let us quickly illustrate how to obtain a bound in terms of
\be
  \norm{a} := \sup_{1\leq i,j\leq n}\abs{a_{i,j}}.
\ee
With the Stein coupling from Variant 3 we have
\be
  \abs{G} \leq 2n\norm{a} =: \alpha, \qquad \abs{D} \leq 2\norm{a} =: \beta.
\ee
We will make use of an auxiliary variable $W''$, which can be constructed so
that it is independent of $(I_1,I_2,\pi(I_1),\pi(I_2))$ and such that
\ben              \label{49}
  \abs{D'} \leq 8\norm{a} =: \beta'.
\ee
Hence, applying Theorem \ref{thm2} with the above random variables and
constants and in addition $\~D = D$ and $\~\beta = \beta$ we easily obtain the
following result. 

\begin{theorem} With $W$ and $\norm{a}$ as above,
\be
  \dk\bklr{\law(W),\N(0,1)}\leq 448 n \norm{a}^3 + 96 \norm{a}.
\ee
\end{theorem}

\begin{proof}
We only need the existence of $W''$ as claimed above; we use the construction of
\cite{Bolthausen1984}. It turns out that it is more convenient to construct $W$
from~$W''$. Let $\tau$ be a uniformly chosen random
permutation of  $[n]$ and let $W'' = \sum_{i}a_{i,\tau(i)}$. Let
$(I_1,I_2,J_1,J_2)$ be random variables independent of $\tau$ such that
$(I_1,I_2,J_1)$ is uniform on~$[n]^3$, such that $J_2$ is uniform on
$[n]\setminus\{J_1\}$ if $I_1\neq I_2$ and such that $J_1= J_2$ if $I_1 = I_2$.
One can now construct a permutation $\pi$ which again has uniform distribution,
and such that 
\be
  \pi(I_1) = J_1,\qquad \pi(I_2) = J_2,
\ee
and such that $\tau$ and $\pi$ differ in at most four positions. Now
$\IE^\tau(G_3(W'_3-W)) = \IE(G_3(W'_3-W))$ (and hence $r_3=0$) follows from the
independence assumption between $\tau$ and $(I_1,I_2,J_1,J_2)$. Note that the
calculations from Variant 3 above still hold, as, given $\pi$, $(I_1,I_2)$ is
uniformly distributed on~$[n]^2$, hence independent of $\pi$ as required.
\end{proof}

\begin{remark}
Note that \cite{Goldstein2005b}, using zero-biasing, obtains
\be
  \dk\bklr{\law(W),\N(0,1)}\leq 1016 \norm{a} + 768 \norm{a}^2.
\ee
\end{remark}

\subsection{Functionals in the classic occupancy scheme} 

Let $m$ balls be distributed independently of each other into $n$ boxes such
that the probability of landing in box $i$ is $p_i$, where $\sum_{i=1}^n p_i =
1$. The literature on this topic is rich; see for example
\cite{Johnson1977}, \cite{Kolchin1978} or \cite{Barbour1992}, but also more
recent results such as \cite{Hwang2008} and \cite{Barbour2009} on local
limits theorems for infinite number of urns. If $\xi_i$ denotes
the number of balls in urn $i$ after distributing the balls, some interesting
statistics can be written in the form
\be
  U = \sum_{i=1}^n h(\xi_i)
\ee
for functions $h:\IZ_+\to\IR$. Examples are
\ba
  h(x) &= \I[x=k]                         
  &&\text{``\# urns with exactly $k$ balls'',}\\
  h(x) &= \I[x> m_0]
  &&\text{``\# urns exceeding a limit $m_0$'',}\\
  h(x) &= \I[x> m_0](x-m_0)
  &&\text{``\# excess balls when urn limit is $m_0$'';}
\ee
see for example \cite{Boutsikas2002}. 

Let us consider here the more general case
\ben                                                          \label{52}
  U = \sum_{i=1}^n h_i(\xi_i)
\ee
for functions $h_i:\IZ_+\to\IR$, $i=1,\dots,n$. Due to the subsequent centering,
we may assume without loss of generality that $h_i(0) = 0$ for all~$i$. 

\def\_#1{\raisebox{0.2ex}{\underline{\kern1ex}}\kern-1ex#1}

\begin{theorem}\label{thm5} Assume the situation as described above. Let $W =
(U-\mu)/\sigma$, where $\mu$ and $\sigma^2$ are the mean and the variance
of~$U$. Define the quantities
\bg
  \norm{h} = \sup_{1\leq i\leq n}\norm{h_i}, 
  \quad \norm{\D h} = \sup_{1\leq i\leq
n}\sup_{j\in\IZ_+}\norm{h_i(j+1)-h_i(j)},\quad
  \-p = \sup_{1\leq i\leq n} p_i,
\ee
and assume that $\norm{h}\geq 1$ and $\norm{\D h}\geq 1$. Then, if 
\ben                                          \label{53}
  1+10m\-p \leq 4\ln(n \norm{h}) \leq \frac{1}{(2\-p)^{1/2}} ,
\ee
we have
\bes
  &\dk\bklr{\law(W),\N(0,1)} \\
  &\qquad \leq 
    \frac{409600n\norm{\D h}^3\ln(n\norm{h})^6(1+\sigma^2/n)}{\sigma^3}
    +\frac{3888\ln(n\norm{h})^2}{\sigma^2},
\ee
\end{theorem}

The constants are large, but explicit. They can be reduced if stronger
conditions than \eq{53} are imposed, such as minimal values for~$n$.
In typical examples, where $\norm{h}\asymp\norm{\D h}\asymp \sigma^2/n\asymp 1$
we obtain a rate of convergence of $\bigo(\ln(n)^6n^{-1/2})$. The correct rate
of convergence for the number of empty urns of equiprobable urns is
$\bigo(n^{-1/2})$, which is best possible and was obtained first
by~\cite{Englund1981}. Although our bound does not yield this rate, it is far
more general. As a consequence of our result we have the following.

\begin{corollary} Sufficient conditions that $U$ as in \eq{52} is
asymptotically normal as $n\toinf$ are
\ba
  (i)  &\enskip\text{$\norm{\D h}$ remains bounded,}
  &(ii) &\enskip m\-p = \bigo(\ln(n)).\\
  (iii) &\enskip \ln(n)^4n^{2/3}/\Var U\tozero,
  &(iv)&\enskip \-p = \bigo(\ln(n)^{-2}),\\
\ee 
\end{corollary}

\begin{proof}[Proof of Theorem~\ref{thm5}]
Let us first state a simple, but key
observation, which will be used in the proof:\smallskip
\begin{paragraph}{\it Fact~A} Assume the situation as stated in the theorem and
let $K\subset[n]$ be arbitrary. Then, the joint distribution of the
balls in the urns of $K$ and $K^c$ is only connected 
through the total number of balls in these respective subsets. That is, given
the number of balls in each urn of $K$ and assuming there are a total of $N$
balls in the urns of $K$, then the balls in $K^c$ are distributed as if $m-N$
balls were distributed among the urns $K^c$ according to the distribution 
$\bklr{{p_i}/{\sum_{k\in K^c} p_k}}_{i\in K^c}$.
\end{paragraph}
\smallskip

We will use Construction~\ref{con6} to construct our Stein coupling. First,
distribute
$m$
balls into $n$ white urns according to $(p_i)_{i\in[n]}$ and denote by
$\xi_j$ the number of balls in white urn~$j$. Fix $i$ and note that 
\ben                                                    \label{54}
  \law(\xi_i)=\Bi(m,p_i).
\ee
Let us now construct a family of random variables $(\xi'_{i,j})_{j\in[n]}$ such
that \eq{37} and \eq{38} are satisfied. 

Assume we have an additional set of $n$ black urns. Independent of all else, let
$\xi'_{i,i}$ have distribution \eq{54} and put that many balls into
black
urn~$i$. We proceed with the remaining $\xi'_{i,j}$, $j\neq
i$.  First, for each $j\neq i$, put $\xi_j$ balls into black urn~$j$.
Let $N_1 = \abs{\xi_i-\xi'_{i,i}}$ be the difference of balls in the
white and black urn~$i$. If $\xi_{i} > \xi'_{i,i}$, distribute an additional
$N_1$ balls into the remaining black urns according to the distribution 
$\bklr{{p_k}/\klr{1-p_i}}_{k\neq i}$. If $\xi_{i} < \xi'_{i,i}$,  remove
instead $N_1$ balls from the remaining
black urns, where the balls are chosen uniformly among all the balls in the
black urns except those in black urn~$i$. It is not difficult to
see that the construction is in fact symmetric, that is,
\be
  \law\bklr{(\xi_i)_{i\in[n]},(\xi'_{i,j})_{j\in[n]}} =
    \law\bklr{(\xi'_{i,j})_{j\in[n]},(\xi_i)_{i\in[n]}}.
\ee
Now, define $U'_i = \sum_{j} h_i(\xi'_{i,j})$ and $W'_i = (U_i-\mu)/\sigma$. 
It is clear that \eq{37} and \eq{38} hold. With $I$ uniformly distributed on
$[n]$ and independent of
all and with $\mu_i := \IE h_i(\xi_i)$, we hence have that 
\be 
  (W,W',G):=(W,W'_I, -n(h_I(\xi_I)-\mu_i)/\sigma)
\ee
is a Stein coupling.

Fix again~$i$. We now construct another family of random variables
$(\xi''_{i,j})_{j\in[n]}$. Denote by $ K_i = \bklg{j\neq i \,:\,
\xi'_{i,j}\neq \xi_{j}}$ the set of indices of those urns for which the number
of balls in the corresponding black urn and white urn differ, not including urn
$i$, and $\-K_i = K_i\cup\{i\}$. Assume that we have a set of $n$ red
urns. First, independently of all else, distribute $N''_2 \sim \Bi(m,\sum_{k\in
\-K_i} p_k)$ balls into the red urns from $\-K_i$ according to the distribution 
$\bklr{{p_j}/{\sum_{k\in \-K_i} p_k}}_{j\in \-K_i}$;
denote by $N_2 = \sum_{j\in \-K_i} X_j $ the number of balls in the white urns
from $\-K_i$, and by $N_3 = \abs{N_2 - N_2''}$ their difference. 
Now, in very much the same way as we constructed the balls in the black urns
from those in the white urns, put $\xi_j$ balls into red urn~$j$ for each
$j\in \-K_i^c$. If $N_2 > N''_2$ distribute another $N_3$ balls into the
red urns from $\-K_i^c$ according to the distribution
$\bklr{{p_j}/{\sum_{k\in \-K_i^c} p_k}}_{j\in \-K^c_i}$,
and if $N_2 < N''_2$
remove $N_3$ balls uniformly among all the balls in the
red urns from $\-K_i$. Now, given $(\xi_j,\xi'_{i,j})_{j\in \-K_i}$, by
construction,
the $(\xi''_{j,i})_{j\in[n]}$ always will have the distribution of
$(\xi_j)_{j\in [n]}$. Hence, $U_i'' := \sum_{j=1}^n h_j(\xi''_{i,j})$ is
independent of $U_i'-U = \sum_{j\in \-K_i}\bklr{h_i(\xi'_{i,j})-h_i(\xi_i)}$ and
$\xi_i$, that is, $W'' = (U_I''-\mu)/\sigma$ is independent of~$GD$. 

We are now in a position to apply  Theorem~\ref{thm2} with $\~D=D$ and~$S=1$. We
clearly have
$r_0=r_1=r_2=r_3 = 0$. Note that
\be
  D  = \frac{1}{\sigma}\sum_{j\in \-K_I}\bklr{h_j(\xi'_{I,j})-h_j(\xi_j)},
  \qquad
  D' = \frac{1}{\sigma}\sum_{j\in \-K'_I}\bklr{h_j(\xi''_{I,j})-h_j(\xi_j)},
\ee
where $\-K_i' = \-K_i \cup K_i'$ with $K_i' = \{k\in K_i^c\,:\, \xi_k''\neq
\xi_k\}$.

Let $C>0$. We have $\abs{\mu_i} \leq m\-p\norm{\D h}$ for every~$i$.
Then, with $\alpha:=nC\norm{\D h}/\sigma + nm\-p\norm{\D h}/\sigma$,
\bes
  \IP[\abs{G}>\alpha]&
  \leq\IP[n\abs{h_I(\xi_I)}/\sigma>\alpha-n\mu_I/\sigma]
  \leq\IP[\abs{h_I(\xi_I)} > C\norm{\D h}] \\
  & \leq \IP[\xi_I > C]
  \leq \IP[\Bi(m,\-p) > C].
\ee
Furthermore, with $\beta := 2C(C+1)\norm{\D h} /\sigma$,
\ba
  \IP[\abs{D} > \beta]
  &\leq \IP[\xi'_I\vee \xi_I> C] + \IP[\abs{D} > \beta,
\xi'_I\vee \xi_I\leq C]\\
  &\leq \IP[\xi'_I\vee \xi_I > C]
   + \IP[\xi'_I\vee \xi_I \leq C, \text{
$\exists i\in K_I$ s.t. $\xi'_i\vee \xi_i>C$}] \\
  &\quad + \IP[\abs{D} > \beta, 
    \xi'_i\vee \xi_i \leq C \,\forall i\in \-K_I],
\intertext{and, noting that the last expression is zero, this is}
 & \leq \IP[\xi'_I> C] + \IP[\xi_I> C]
 +2C\IP[\Bi(m,\-p/(1-\-p))>C]\\
 & \leq 2\IP[\Bi(m,\-p)> C] + 2C\IP[\Bi(m,2\-p) >C].
\ee
Lastly, with $\beta' := 2C(C+1)^2\norm{\D h} /\sigma$,
\ba
  &\IP[\abs{D'} > \beta']\\
  &\leq \IP[\xi'_I\vee \xi_I> C] 
    + \IP[\abs{D'} > \beta',\xi'_I\vee \xi_I  \leq C]\\
  &\leq \IP[\xi'_I\vee \xi_I> C] 
    + \IP[\xi'_I\vee \xi_I  \leq C, \ssum_{j\in \-K_I}\xi''_j\vee \ssum_{j\in
\-K_I}\xi_j> C(C+1)]\\
   &\quad+ \IP\bkle{\abs{D'} > \beta',\xi'_I\vee \xi_I  \leq C, 
        \ssum_{j\in \-K_I}\xi''_j\vee \ssum_{j\in \-K_I}\xi_j\leq C(C+1)}\\
  &\leq \IP[\xi'_I\vee \xi_I> C] 
    + \IP[\xi'_I\vee \xi_I  \leq C, \ssum_{j\in \-K_I}\xi''_j\vee \ssum_{j\in
\-K_I}\xi_j> C(C+1)]\\
   &\quad+ \IP\bkle{\xi'_I\vee \xi_I  \leq C, 
        \ssum_{j\in \-K_I}\xi''_j\vee \ssum_{j\in \-K_I}\xi_j\leq C(C+1), 
    \text{$\exists i\in \-K'_I$ s.t. $\xi''_i\vee \xi_i>C$}}\\
   &\quad+ \IP\bkle{\abs{D'} > \beta',\xi'_I\vee \xi_I  \leq C, 
%        \ssum_{j\in \-K_I}\xi''_j\vee \ssum_{j\in \-K_I}\xi_j\leq C(C+1),
  \xi''_i\vee \xi_i \leq C \,\forall i\in \-K'_I},
\intertext{and, noting that the last expression is zero, this is}
  &\leq 2\IP[\Bi(m,\-p) > C] + \IP[\Bi(m,(C+1)\-p/(1-\-p)) > C(C+1)]\\
    &\quad +2C(C+1)\IP[\Bi(m,\-p/(1-C(C+1)\-p)>C]\\
  &\leq 2\IP[\Bi(m,\-p) > C] + \IP[\Bi(m,2(C+1)\-p > C(C+1)]\\
    &\quad +2C(C+1)\IP[\Bi(m,2\-p)>C]
\ee
if 
\ben                                          \label{55}
  C(C+1)\-p \leq 1/2.
\ee

Now, from the Chernoff bound we have that, for every $\eps\geq0$,
\be
  \IP[\Bi(n,p)>(1+\eps)np]\leq \exp\bbbklr{-\frac{\eps^2}{2+\eps}np},
\ee
which implies that
\ben                                            \label{56}
  \IP[\Bi(n,p)>x] \leq e^{-x/2}
\ee
as long as~$x>5np$.
Choose $C=4\ln(n \norm{h})-1$; this implies \eq{55} under Condition \eq{53}.
Also under Condition~\eq{53}, \eq{56} can be used to obtain
\ba
  \IP[\abs{G}>\alpha] & \leq \IP[\Bi(m,\-p) > C] \leq \frac{2}{n^2\norm{h}^2},\\
  \IP[\abs{D}>\beta]  & \leq 2\IP[\Bi(m,\-p)> C] + 2C\IP[\Bi(m,2\-p) >C] \\
  &\leq \frac{4+16\ln(n\norm{h})}{n^2\norm{h}^2}\\
  \IP[\abs{D'}>\beta']  & \leq 2\IP[\Bi(m,\-p) > C] + \IP[\Bi(m,2(C+1)\-p >
C(C+1)]\\ &\quad + 2C(C+1)\IP[\Bi(m,2\-p)>C]\\
  &\leq \frac{6+64\ln(n\norm{h})^2}{n^2\norm{h}^2}.
\ee
Now, Theorem~\ref{thm2}, with
$\alpha$, $\~\beta:=\beta$ and $\beta'$ as above, and the rough bounds
\be
  \abs{G}\leq \frac{2n \norm{h}}{\sigma},\qquad \abs{D}\leq
        \frac{2n\norm{h}}{\sigma}.
\ee
yields
\bes
  &\dk\bklr{\law(W),\N(0,1)} \\
  &\leq 12(\alpha\beta+1)\beta'+8\alpha\beta^2\\
  &\quad+\frac{12n^2\norm{h}^2}{\sigma^2}\bklr{\IP[\abs{G} >
\alpha]+\IP[\abs{D} > \beta]+\IP[\abs{D'} > \beta']}\\
  &\leq \frac{48n(C+1)^5\norm{\D h}^3(C + m\-p +\sigma^2/n)
  }{\sigma^3}
  +\frac{32n(C+1)^4\norm{\D h}^3(C + m\-p)}{\sigma^3}\\
  &\quad+\frac{12\bklr{
    12+16\ln(n\norm{h})+64\ln(n\norm{h})^2}}{\sigma^2}\\
  &\leq \frac{80n(C+1)^5\norm{\D h}^3(C + m\-p +\sigma^2/n)
      }{\sigma^3}
  +\frac{144\bklr{
    1+11\ln(n\norm{h})^2}}{\sigma^2},
\ee
which, after plugging in $C$ and with some further straightforward manipulations
and estimates, yields the final bound.
\end{proof}

\subsection{Neighbourhood statistics of a fixed number of uniformly distributed
points}

Consider the space $\%J = [0,n^{1/d})^d$ for some integer $d\geq 1$ and let
$X_1,\dots,X_n$ be $n$ i.i.d.\ points uniformly distributed on~$\%J$. Let $\psi$
be a measurable real-valued functional defined on all pairs $(x,\%X)$ where
$x\in\%J$ and $\%X\subset \%J$ is a finite subset. Such statistics have been
investigated in many places for specific choices of~$\psi$. If the number of
points is Poisson distributed, approaches using local dependence can be
successfully applied. But if the number of points is fixed, besides the
local dependence also global weak dependence has to be taken into
account. 

Assume that $\psi$ is
translation invariant, i.e.\ $\psi(x+y,\%X+y) = \psi(x,\%X)$, where we
assume the torus convention for translations. Assume that $\psi$ has
influence radius $\rho$, i.e.\ for each $x\in\%J$ and each $\%X\subset \%J$ we
have
\be
  \psi(x,\%X) = \psi(x,\%X\cap B_\rho(x)),
\ee
where $B_\rho(x)$ is the closed sphere of radius $\rho$ and center $x$ under
the toroidal Euclidean metric. To avoid self-overlaps, assume
$\rho<\ahalf n^{1/d}$. Let
$\%X :=
\{X_1,\dots,X_n\}$, define $U = \sum_{x\in\%X} \psi(x,\%X)$, $\sigma^2 =
\Var U$  and $W = U/\sigma$, where we assume that $\IE \psi(X_1,\%X)=
0$. 

\begin{theorem}\label{thm6} Assume $W$ is defined as above and assume also
that $\rho \geq \pi^{-1/2}\Gamma(1+d/2)^{1/d}$. Then there is a universal
constant $C_d$ only depending on $d$ such that
\be
  \dw\bklr{\law(W),\N(0,1)}\leq 
  % 3380\bklr{1+4^d 14+3^d 17}^2\pi^{6/d}/\Gamma(1+1/d)^6
    \frac{C_d\norm{\psi}^3\rho^{6d}n}
{\sigma^3}.
\ee
\end{theorem}

In the case where $d$, $\rho$ and $\norm{\psi}$ remain fixed as $n\toinf$ we
have from \cite{Penrose2001} that $\sigma^2 \asymp n$ as long as
$\Var(\psi(X_1,\%X))>0$; in that case, Theorem~\ref{thm6} gives the best
possible order $n^{-1/2}$ for the Wasserstein metric.

After constructing an extended Stein coupling, the proof of the theorem
essentially amounts to bounding the third moments of some mixed binomial
distributions. Constants could be easily extracted from the proof but with
the rough bound given here they are too large to be of practical use. 

\begin{proof}[Proof of Theorem~\ref{thm6}] We first make a simple, but key
observation, which will be used 
in the proof:\smallskip
\begin{paragraph}{\it Fact~B} Let $\%Y$ be a fixed number of points
distributed uniformly and independently of each other on an open subset
$\%J'\subset \%J$. For any other open subset $U\subset\%J$, the
joint distribution of points of $\%Y$ on $U$ and $U^c$ is only connected through
the number of points in these two respective subsets via the equation
$\abs{U\cap\%Y}+\abs{U^c\cap\%Y} = \abs{\%Y}$. That is,
given the total number and the number of points in one of the subsets, $U$,
say, possibly along
with their locations on $U\cap\%J'$, the remaining points in the subset $U^c$
are then distributed uniformly and independently of each other on $U^c\cap\%J'$.
\end{paragraph}
\smallskip

We first construct a Stein coupling according to Construction~\ref{con3}. Note
that $W_i'$ constructed below will not have the same distribution as~$W$. Let to
this end $I=i$
and $X_i = x_i$ be given. Let 
\be
  N_{i,1} = \bklr{\%X\cap B_\rho(x_i) }\setminus \{x_i\}
\ee
be the points within radius $\rho$ of $x_i$, excluding~$x_i$. Let $N_{i,2}$ be
an $\abs{N_{i,1}}$ number of points uniformly distributed on
$B_\rho^c(x_i)$. Define the new set 
\be
  \%X_i' = \bklr{\%X\cap B_\rho^c(x_i)} \cup N_{i,2};
\ee
that is, remove $x_i$ from $\%X$ and replace the neighbouring points
of $x_i$ by the new points~$N_{i,2}$. Note that $\abs{\%X_i'} = n-1$. Now, let
\be
  U'_i = \sum_{x\in \%X'_i}\psi(x,\%X'_i).
\ee
As $\psi(x_i,\%X) = \psi(x_i,\%X\cap B_\rho(x_i))$ and using Fact~B, we have
that $S'_i$ is
independent of~$\psi(x_i,\%X)$.

Randomising now over $I$ and $X_I$, define $U' = U_I'$, $W' = U_I'/\sigma$ and
$G = - n\psi(X_I,\%X)$. Then, from Construction~\ref{con3}, we have that 
$(W,W',G)$ is a Stein coupling. Indeed, using the above mentioned
independence, 
\be
    \IE\bklg{\psi(X_i,\%X)\bmid U_i'} = 0,
\ee
and hence \eq{31} is satisfied.

We now construct an extension $W''$ of the basic Stein coupling, which will be
independent of $G$ and~$U'-U$. Let again $I=i$ and $X_i = x_i$ be given. First,
we define some further sets. Define 
\be 
  M_{i,1} = B_{2\rho}(x_i) \cup \bigcup_{x\in N_{i,2}}B_\rho(x)
\ee
and denote by $N_{i,3} = \%X \cap (M_{i,1} \setminus B_\rho(x_i))$ all the
points of $\%X$ which remained fixed during the perturbation, but whose values
of $\psi$ were potentially affected, either because of points being removed or
points being added in their neighbourhood. We can now write
\besn                                                       \label{57}
  \D_i := U'_i - U 
   &= \sum_{x\in N_{i,2}}\psi(x,\%X_i')
    +\sum_{x\in N_{i,3}}\bklr{\psi(x,\%X_i') - \psi(x,\%X)} \\
    &\qquad -\sum_{x\in N_{i,1}} \psi(x,\%X) -  \psi(x_i,\%X).
\ee

Let 
\be 
  M_{i,2} =
    B_{3\rho}(x_i)\cup\bigcup_{x\in N_{i,2}} B_{2\rho}(x).
\ee
Clearly, the values of $\D_i$ and $\psi(x_i,\%X)$ are determined by the points
$\%X\cap M_{i,2}$ and $\%X'\cap M_{i,2}$ only. Given the points $\%X\cap
M_{i,2}$ and $\%X'\cap M_{i,2}$, we have from Fact~A that the remaining
$n-\abs{\%X\cap M_{i,2}}$ points of $\%X$ are distributed uniformly and
independently of each other on~$M_{i,2}^c$. Hence we can use them for our
construction of $W''$ as follows. First, let $N_{i,4}$ be a
$\Bi(n,\Vol(M_{i,2})/n)$ number of points distributed independently and
uniformly on~$M_{i,2}$. Then, let $N_{i,5}$ be the remaining $n - \abs{N_{i,4}}$
points on $M_{i,2}^c$ distributed as follows. Let $P(k)$ be $k$ random points
distributed uniformly on $M_{i,2}^c$ if $k>0$, and let, for $k<0$, $P(k)$ be
$\abs{k}$ randomly chosen points from $\%X\cap M_{i,2}^c$ (without replacement);
set $P(0) = \emptyset$. Let $P := P(\abs{\%X\cap
M_{i,2}}-\abs{N_{i,4}})$ and define
\be
  N_{i,5} = 
    \begin{cases}
        \klr{\%X\cap M_{i,2}^c} \setminus P
            &\text{if $\abs{N_{i,4}}>\abs{\%X\cap M_{i,2}}$,}\\
        \klr{\%X\cap M_{i,2}^c} \cup P
            &\text{if $\abs{N_{i,4}}\leq\abs{\%X\cap M_{i,2}}$;}   
     \end{cases}
\ee
that is, we let $N_{i,5}$ be the points $\%X\cap M_{i,2}^c$ and add or remove as
many points (that is the points $P$) as needed so that
$\abs{N_{i,4}}+\abs{N_{i,5}} = n$. Now, setting $\%X''_i = N_{i,4}\cup N_{i,5}$,
we clearly have that $U''_i = \sum_{x\in\%X''_i} \psi(x,\%X''_i)$ is independent
of $(\psi(x_i,\%X),\D)$ as the conditional distribution of $\%X''_i$ given
$(\%X\cap M_{i,2},\%X'_i\cap M_{i,2})$ always equals to $\law(\%X)$ by
construction and $(\%X\cap M_{i,2},\%X'_i\cap M_{i,2})$ determines
$(\psi(x_i,\%X),\D_i)$ as mentioned before. Randomizing over $I$ and $X_I$, set
$U'' = U''_I$ and $W'' = U''/\sigma$. Thus, $(W,W',G, W'')$ satisfies the
assumptions of Corollary~\ref{cor3} when setting~$\~D = D$. It remains to
find
the corresponding quantities $A$ and~$B$.

Define the sets
\bes
M_{i,3} & =  B_{4\rho}(x_i)\cup\bigcup_{x\in N_{i,2}} B_{3\rho}(x),\\
M_{i,4} & =  \bigcup_{x\in P} B_\rho(x),\\
\ee
and let
\ba
  Y_1 & := \abs{\%X\cap (B_\rho(x_i) \setminus\{x_i\})} = \abs{N_{i,1}}
= \abs{N_{i,2}}\\
  Y_2 &:= \abs{\%X\cap (M_{i,2}\setminus B_\rho(x_i))} = \abs{\%X\cap
        M_{i,2}} -1 - Y_1, \\
  Y_3 &:= \abs{\%X\cap (M_{i,3}\setminus M_{i,2})} = 
    \abs{\%X\cap M_{i,3}}- 1 - Y_1 - Y_2, \\
  Y_4 &:= \abs{N_{i,4}} = \abs{\%X_i''\cap M_{i,2}}.
\ee
Let $\kappa_\rho := \Vol(B_\rho(0))$. 
Then the following statements are straightforward:
\ba
  (i)\enskip&\law(Y_1) =  \Bi(n-1,\kappa_\rho/n),\\
  (ii)\enskip&\law(Y_2+Y_3|N_{i,2})=
  \Bi\bbklr{n-1-Y_1,\tsfrac{\Vol(M_{i,3}\setminus
  B_\rho(x_i))}{(n-\kappa_\rho)}},\\
  (iii)\enskip&\law(Y_4|N_{i,2})= \Bi\bklr{n,\Vol(M_{i,2})/n}, \\
  (iv)\enskip& \text{$Y_4\indep (Y_2,Y_3)$ given $N_{i,2}$,}\\
  (v)\enskip&\abs{P} \leq \abs{\%X\cap M_{i,2}} + \abs{\%X_i''\cap M_{i,2}} =
      1+Y_1+Y_2 + Y_4,\\
  (vi)\enskip&\abs{\%X_i''\cap M_{i,3}} \leq Y_4 + Y_3 + \abs{P} \leq
  1+Y_1+Y_2+Y_3+2Y_4,\\
  (vii)\enskip&\abs{(\%X_i''\setminus P)\cap M_{i,4}} \leq  
  \abs{\%X_i''\cap M_{i,3}} + Y_5 \quad\text{and}\\
  &\abs{(\%X\setminus P)\cap M_{i,4}} \leq  \abs{\%X_i\cap M_{i,3}}+ Y_5,\\
  &\text{for some $Y_5$ with  $\law(Y_5|P,N_{i,2},N_{i,4})\sim
\Bi\bbklr{n-Y_1-1,\tsfrac{\Vol(M_{i,4})}{n-\Vol(M_{i,3})}}$},
\ee
The bound in $(vi)$ follows from the fact that the difference between the points
$\%X\cap(M_{i,3}\setminus M_{i,2})$ and $\%X''\cap(M_{i,3}\setminus M_{i,2})$ is
at most the points from~$P$. The somewhat rough bound $(vii)$ is due to the fact
that the neighbourhoods of the points in $P$ may overlap with~$M_{i,3}$. 

Let now $C$ be a constant that may depend only on the dimension $d$ but can
change from formula to formula.  From representation \eq{57}, we easily
obtain
\be
  \abs{D_i}  \leq 
    \norm{\psi}\bklr{\abs{N_{i,2}} + 2\abs{N_{i,3}}+
    \abs{N_{i,1}}+1}/\sigma
  \leq C\norm{\psi}(Y_1+Y_2+1)/\sigma.
\ee
Now, we can write $\D'_i := U_i''-U$ as 
\bes
  \D'_i & = \sum_{x\in  \%X''_i\cap M_{i,3}} \phi(x,\%X_i'') 
      - \sum_{x\in \%X \cap M_{i,3}} \phi(x,\%X) \\
       &\quad  + \sum_{x\in \%X''_i \cap  (M_{i,4}\setminus M_{i,3})}
        \phi(x,\%X''_i) 
      - \sum_{x\in \%X \cap (M_{i,4}\setminus M_{i,3})} \phi(x,\%X),
\ee
hence
\bes
  \abs{D_i'} &\leq \norm{\psi} \bklr{\abs{\%X''_i \cap M_{i,3}} + \abs{\%X \cap
M_{i,3}} +
      \abs{\%X''_i\cap M_{i,4}} + \abs{\%X \cap M_{i,4}} } /\sigma\\
      & \leq C \norm{\psi} Z /\sigma,
\ee
where $Z := 1+Y_1+Y_2+Y_3 + Y_4 + Y_5$.

To estimate the third absolute moment of $Z$ note first that if $Y\sim\Bi(m,p)$
with $mp\geq 1$ we have
\be
  \IE Y^3 \leq 5m^3p^3.
\ee
Note that the assumption $\rho\geq \pi^{-1/2}\Gamma(1+d/2)^{1/d}$ implies that
$\IE Y_1=\kappa_\rho\geq 1$ because $\kappa_\rho = \rho^d \kappa_1 =
\rho^d\pi^{d/2}/\Gamma(1+d/2)$.

As the cubic function is convex on the non-negative half line, we have for any
non-negative numbers $a_1,\dots,a_m$ that
\be
  (a_1+\dots +a_m)^3 \leq m^2(a_1^3+\dots+a_m^3).
\ee
From $(i)$ we immediately obtain 
\ben                                                            \label{58}
  \IE Y_1^3 \leq 5\kappa^3_\rho.
\ee
From $(ii)$--$(iv)$ we have that, given $Y_1$, $Y_2+Y_3+Y_4$ is
stochastically dominated by $\Bi(2n,2\Vol(M_{i,3})/n)$ which is
further stochastically dominated by
$\Bi(2n,(2\kappa_{4\rho}+2\kappa_{3\rho}Y_1)/n)$, where we define
$\Bi(m,p):=\Bi(m,1)$ if $p>1$; hence, using this and \eq{58},
\ben                                                            \label{59}
  \IE\klr{Y_2+Y_3+Y_4}^3 
  \leq C\IE\klr{\kappa_{4\rho}+\kappa_{3\rho}Y_1}^3
%   &\leq C\kappa_\rho^3\IE\klr{1+Y_1}^3\\
%   &\leq C\kappa_{\rho}^3 + C\kappa_{\rho}^3\IE Y_1^3\\
  \leq C(\kappa_{\rho}^3 + \kappa_{\rho}^6).
\ee
Note now that, given $P$, $N_{i,2}$ and $N_{i,4}$, we have from
$(vii)$ that $\law(Y_5)$ is stochastically dominated by
$\Bi(n-1-Y_1,\kappa_\rho\abs{P}/(n-\kappa_{4\rho}-\kappa_{3\rho}Y_1)_+)$ which,
by~$(v)$, can be dominated by 
$\Bi(n-1-Y_1,\kappa_\rho(1+Y_1+Y_2+Y_4)/(n-\kappa_{4\rho}-\kappa_{3\rho}
Y_1)_+)$; hence, using this, \eq{58} and
\eq{59}, we obtain
\bes
  \IE Y_5^3 & \leq
C\IE\bbbklg{(n-1-Y_1)\bbbklr{1\wedge\frac{\kappa_\rho^3(1+Y_1+Y_2+Y_4
)}{(n-\kappa_{4\rho}-\kappa_{3\rho}Y_1)_+}}}^3\\
  & \leq n\IP[Y_1>cn] +
C\kappa_\rho^3\IE\klr{1+Y_1+Y_2+Y_4
}^3
\ee
for $n$ large enough and $c$ such that
${n-\kappa_{4\rho}-cn\kappa_{3\rho}}\geq n/2$, e.g.\ $ c =
1/(4\kappa_{3\rho})$ and $n\geq
4\kappa_{4\rho}$. From the Chernoff-Hoeffding inequality we obtain that
$\IP[Y_1\geq cn]\leq 2^{-cn}$ for $cn$ large enough. Hence, $n\IP[Y_1\geq
nc]\leq C \leq C\kappa_\rho^3$ and 
\besn                                                           \label{60}
  \IE Y_5^3 
   &\leq C\kappa_\rho^3\bklr{1+\IE Y_1^3 + \IE(Y_2+Y_3+Y_4)^3}\\
   &\leq C(\kappa_\rho^3
     + \kappa_\rho^6 +
     \kappa_{\rho}^9).
\ee

Putting \eq{58}--\eq{60} together we obtain
\bes
  \IE Z^3 &\leq C (1+\IE Y_1^3+ \IE(Y_2+Y_3 + Y_4)^3 + \IE Y_5^3)\\
   &\leq C (1+\kappa_{\rho}^3+\kappa_\rho^6 +\kappa_{\rho}^9)
   \leq C\kappa_{\rho}^9,
\ee
as $\kappa_{\rho}\geq 1$. Hence
\be
  \IE \abs{D'}^3 = \frac{1}{n}\sum_{i=1}^n\IE\abs{D_i'}^3
   \leq C\norm{\psi}^3\kappa_\rho^9/\sigma^3  =: A^3.
\ee

Furthermore, we have
\ben                                                    \label{61}
  \IE\abs{G}^3 \leq n^3\norm{\psi}^3/\sigma^3 =: B^3.
\ee
Now, Corollary~\ref{cor3} yields 
\be
  \dw\bklr{\law(W),\N(0,1)}\leq 5A^2B \leq
    Cn\norm{\psi}^3\kappa_\rho^6/\sigma^3.
\ee
By   recalling that $\kappa_\rho = \rho^d\kappa_1$ we obtain the
final bound.
\end{proof}

\subsection{Susceptibility and related statistics in the sub-critical Erd\H
os-R\'enyi random graph} \label{sec13}

Let $H$ be a graph (here, we use the letter `$H$' for graphs as the
letter `$G$' is used in the context of Stein couplings).
Then, \emph{susceptibility} $\chi(H)$ is defined to be the expected size of the
component containing a uniformly chosen random vertex. That is, if $C_i\subset
H$, $i=1,\dots,K$ are the $K$ maximal subgraphs of the graph $H$, 
\ben                                                        \label{62}
  \chi(H) = \sum_{i=1}^K \frac{\abs{C_i}}{n}\abs{C_i},
\ee
where $\abs{C_i}$ denotes the number of vertices in subgraph~$C_i$.
Using a different normalisation in \eq{62} we can also write
\ben                                                        \label{63}
 \sum_{i=1}^K \frac{\abs{C_i}}{n}\frac{\abs{C_i}}{n} =
\frac{1}{n}\chi(H),
\ee
which, hence, is the probability that, for a given graph $H$, two randomly
chosen
vertices are in the same component, or, equivalently, connected through a
path. Note that, as $H$ is random, the probability of being connected is
therefore itself a random quantity. 

Using a martingale CLT it was proved by \cite{Janson2008} that, if
$H$ is a sub-critical \ER\ graph, $\chi(H)$ is asymptotically
normal. However, it is well known that standard martingale CLTs will often not
yield optimal rates of convergence with respect to the Kolmogorov metric (see
e.g.\ \cite{Bolthausen1982a}). As often the actual dependence structure under
consideration is much weaker than in a worst-case scenario, \cite{Rinott1998}
provide bounds that incorporate expressions to measure the dependency and 
which yield optimal bounds in many settings, but the involved quantities
are typically hard to calculate. Let us derive here some bounds without using
martingales and which are optimal up to some additional logarithmic factors.

Let us consider here more general statistics that depend only on the properties
of
the components of two randomly chosen vertices. For a vertex $i \in V(H)$
denote by $C(H,i)\subset H$ the component (i.e. the maximal subgraph)
containing~$i$. Let $h = h(H,i,j)$ be a function that is determined by $i$ and
$j$ and the two components that contain $i$ and $j$, respectively, that is,
\ben                                                    \label{64}
  h(H,i,j) = h(H\cap (C(H,i)\cup C(H,j)),i,j),\qquad
  \text{for all $H$, $i$ and $j$,}
\ee
and which is symmetric in the sense that
\ben                                                     \label{65}
  h(H,i,j)=h(H,j,i),\qquad
  \text{for all $H$, $i$ and $j$.}
\ee
Then define
\ben                                                      \label{66}
  U(H) := \sum_{i\in V(H)}\sum_{j\in V(H)} h(H,i,j).
\ee
We can recover susceptibility \eq{62} from \eq{66} (up to a normalising
constant) simply by choosing
\ben                                              \label{67}
   h(H,i,j) = \I[\text{$i$ and $j$ are in the same component}].
\ee
As there is little hope for a normal limit for general $h$ we consider
functions that are non-zero only if the two random vertices are in
the same component, that is,
\ben                                                  \label{68}
  C(H,i)\neq C(H,j) \quad\implies\quad h(H,i,j) = 0.
\ee
Hence, under \eq{68}, we can write \eq{66} also as 
\ben                                                      \label{69}
  U(H) = \sum_{i\in V(H)}\sum_{j\in C(H,i)} h(H,i,j).
\ee
Dependence on just one vertex and its component is, of
course, also covered as a special case $h(H,i,i) = h_0(H,i)$ and $h(H,i,j) =
0$ if~$i \neq j$. For example $h_0(H,i)
= \I[\abs{C(H,i)}=1]$ yields the total number of singletons, or, more generally,
choose $h_0(H,i) = \I[C(H,i)>m_0]$  to obtain the number of vertices that are in
a component of size larger than~$m_0$. We can also recover \eq{62} by
setting
$h_0(H,i) = \abs{C(H,i)}^2/n$, however this function is not Lipschitz in the
size of the component. Other quantities of interest may be obtained, such as
\ba
  h(H,i,j) 
  & = \I[\text{$i$ and $j$ are connected, but not further than $m_0$
apart}],\\
   h(H,i,j) 
   & = \I[\text{$i$ and $j$ are in the same cycle}],\\
   h(H,i,j) 
   & = \I[\text{$i$ and $j$ are connected}]/
 	      \abs{C(H,i)}.
\ee
Note that by summing uniformly over
all the vertices, the components are size-biased. However, we can `de-bias'
as follows: for a given $h_0(H,i)$ of interest,
$h^u_0(H,i) = h_0(H,i)/\abs{C(H,i)}$ gives the corresponding unbiased result,
where now summing is done uniformly among the components. As $\abs{C(H,i)}\geq
1$, we have $\norm{\D h^u_0}\leq\norm{\D h_0}$. However, averaging is only
possible with respect to the expected number of components, that is, $h^u_0(H,i)
= \frac{h_0(H,i)}{\abs{C(H,i)}\IE K }$, not the actual number, as the function
$\frac{h_0(H,i)}{\abs{C(H,i)}K}$ does not satisfy~\eq{69}.

We have the following theorem.

\begin{theorem}\label{thm7} Let $U$ be as in \eq{66} for some function $h$
satisfying 
\eq{64}, \eq{65} and ~\eq{68}. Let $W = (U(H)- \mu)/\sigma$,
where $\mu=\IE U(H)$ and $\sigma^2 = \Var U(H)$. Let
\be
  \norm{h} = \sup_{H}\sup_{i,j\in[n]}\abs{h(H,i,j)}\geq 1
\ee
and, with $H^{k,l}$ denoting the graph where an additional edge is added
between $k$ and $l$ if $k\neq l$,
\be
  \norm{\D h} = \sup_{i,j,k,l\in [n]}\sup_{H}\babs{h(H^{k,l},i,j) - h(H,i,j)}
  \geq 1.
\ee
Assume that $0\leq h(H,i,j)\leq 2\norm{\D h}$ 
whenever $H$ is such that the components which contain $i$ and~$j$,
respectively, are singletons. 
Let $H$ be an \ER\ random graph with $n$ vertices and edge probabilities
$\lambda/n<1/n$. 
Then, there is a universal constant $K>0$ such that
\bes
  & \dk\bklr{\law(W),\N(0,1)}  \\
  & \quad \leq  K\bbbklr{
\frac{n\log(n\norm{h})^{11}\norm {\D
h}^3(1+1/a^2+\sigma^2/n)}{\lambda a^{11}\sigma^3}
    +\frac{e^a\ln(n\norm{h})^2}{\lambda a^2\sigma^2}}
\ee
whenever
\ben                                                    \label{70}
    a := \lambda -1 - \log \lambda \leq 4\ln(n\norm{h}).
\ee
\end{theorem}

Let now $\lambda<1$ be fixed and let $h$ be as in \eq{67} to obtain
susceptibility (up a normalising constant). Clearly, $\norm{h} = 1$ and
$\norm{\D h} = 1$. From \cite{Janson2008} we have that $\Var U(H)\sim 2\lambda
n(1-\lambda)^{-5}$, hence Theorem~\ref{thm7} yields a Kolmogorov bound of order
$\log(n)^{11}/\sqrt{n}$.

\def\iscon{\sim}
\begin{proof}[Proof of Theorem~\ref{thm7}] 
% We first start with a construction that leads to an exchangeable
% pair and will motivate the later actual, different construction of the Stein
% coupling. 

We will make use of Construction~\ref{con3} to obtain a Stein coupling.
We consider throughout the vertex set~$[n]$. Let
$C_V(H,i)$ denote the vertices of~$C(H,i)$. If $e=\{k,l\} \subset V(H)=[n]$ is a
potential edge in a graph $H$ on $[n]$ and
$K\subset [n]$, write $e\iscon K$ if $k\in K$
or $l\in K$ (or both). Write $e\not\iscon K$ if $k\notin K$ and $l\notin
K$.

Let now $\xi = (\xi_{\{i,j\}})$, where $i,j\in[n]$ and $i\neq j$, be an i.i.d.\
family of $\Be(\lambda/n)$ distributed random variables. Let $H$ be the graph on
the
vertex set $[n]$ with edge set $\xi$, that is, $e$ is an edge in $H$ iff
$\xi_e=1$. Let
$H^*$ be an independent copy of~$H$.  Note that, given
$C(H,i)$, we have that $(\xi_e)_{e\not\iscon C_V(H,i)}$ is a family of
i.i.d.\ $\Be(\lambda/n)$ random variables. Define the random graph $H'_{i}$
through
\be
e\in E(H_{i}')\quad:\iff 
   \vcenter{\hsize=0.5\hsize
    \vbox{$e\in E(H^*)$ and 
          $e\iscon C_V(H,i)$,$\enskip$ or}
    \vbox{$e\in E(H)$ 
      and $e\not\iscon C_V(H,i)$.}}
\ee
As $H'_{i}$, given $C(H,i)$, is always an unconditional \ER\ 
random graph by construction, we have that $C(H,i)$ is independent of the
graph~$H'_{i}$. 

Now, let $I$ be uniformly and independently distributed on~$[n]$. Set
$H'=H'_{I}$, $U = U(H)$ and $U' = U(H')$. Furthermore,
$W = (U-\mu)/\sigma$, $W' = (U'-\mu)/\sigma$ and 
\be
  G = - \frac{n}{\sigma}\bbbklr{\sum_{j\in C(H,I)}h(H,I,j)-\mu_{I}},
\ee
where $\mu_{i} := \IE \sum_{j\in C(H,i)}h(H,i,j)$. Then, it is easy to see
that \eq{31} is satisfied and hence 
we have a Stein coupling.

Let now
\be
    V_i := %C_V(H,i)\cup 
        \bigcup_{k\in C_V(H,i)} C_V(H_i',k)
\ee
be all those vertices of the graph $H$ whose components are affected
by changing from $H$ to~$H'_i$; that is, the vertices $C_V(H,i)$ themselves
on one hand and, on the other hand, the vertices $[n]\setminus C_V(H,i)$
that got connected in $H'_i$ to at least one of the vertices~$C_V(H,i)$.

Then we can write
\be
  \D_{i} := U(H_i') - U(H) = 
  \sum_{k\in V_i}\sum_{l\in C_V(H,k)} \bklr{h(H_i',k,l)-h(H,k,l)}.
\ee
Note that the subgraphs $H\cap V_i$ and $H_i'\cap V_i$
determine the value of~$\D_i$. Hence, similarly as before, let
$H^{**}$ be an independent copy of $H$ and define the graph $H_i''$ by 
\be
e\in E(H_i'')\quad:\iff 
   \vcenter{\hsize=0.5\hsize
    \vbox{$e\in E(H^{**})$ and 
          $e\iscon V_i$,$\enskip$ or}
    \vbox{$e\in E(H)$ 
      and $e\not\iscon V_i$.}}
\ee
With the same argument as before, $H''_i$ is independent of
$H\cap V_i$ and $H_i'\cap V_i$. Define 
\be
  V'_i = %V_i\cup 
        \bigcup_{k\in V_i} C_V(H_i'',k)
\ee
to be all those vertices of the graph $H$ whose components are affected
by changing from $H$ to~$H''_i$. We can write
\be
  \D'_i := U(H''_i) - U(H) = 
      \sum_{k\in V_i'}\sum_{l\in C_V(H,k)} \bklr{h(H_i'',k,l)-h(H,k,l)}.
\ee

We are now in a position to apply Theorem~\ref{thm2}. As $(W,W',G)$ is a Stein
coupling we have~$r_0 = 0$. Setting $\~D = D$ and $S=1$ we have $r_1=r_2=0$, and
by construction of $H''_i$, that~$r_3 = 0$. 

Let $C>0$ to be chosen later. From \cite[p.~38]{Durrett2007} we have that 
\ben                                                        \label{71}
  \IP(\abs{C(H,1  )}\geq k) \leq \frac{e^{-ak}}{\lambda}.
\ee
Hence, we obtain the simple estimate
\be
    \IE\abs{C(H,1)}^2= \int_{0}^\infty\IP[\abs{C(H,1)}>\sqrt{x}\,] dx
    \leq \frac{2}{\lambda a^2}.
\ee
Note also that, if $j\in C_V(H,i)$,
\be
  \abs{h(H,i,j)}\leq \norm{\D h}\abs{C(H,i)}.
\ee
This yields $\abs{\mu_{i}}\leq \norm{\D h}\IE\abs{C(H,i)}^2\leq \norm{\D
h}\frac{2}{\lambda a^2}$. Furthermore, if $j\in V_i\setminus C_V(H,i)$, then
\be
  \abs{C(H,j)}\leq \abs{C(H_i',j)}.
\ee
With $\alpha :=
\frac{n}{\sigma}\norm{\D h}(C^2+\frac{2}{\lambda a^2})$ we have 
\ba
  \IP[\abs{G}>\alpha] 
  & \leq
  \IP[\abs{C(H,I)}> C] + \IP[\abs{G}>\alpha, \abs{C(H,I)}\leq C] 
\intertext{and, noticing that the last term is zero, this is}
  & = \IP[\abs{C(H,1)}> C].
\ee
Furthermore, with $\beta := 2C^4\norm{\D h}/\sigma$,
\ba
  & \IP[\abs{D}>\beta] \\
  & \leq \IP[\abs{C(H,I)}>C] + \IP[\abs{D}>\beta,\abs{C(H,I)}\leq C]\\
  & \leq \IP[\abs{C(H,1)}>C] \\
  & \quad + \IP[\abs{C(H,I)}\leq C, \text{$\exists j\in C_V(H,I)$ s.t.\
$\abs{C(H'_I,j)}>C$}] \\
  &\quad +\IP[\abs{D}>\beta,\abs{C(H,I)}\leq C,\text{
$\abs{C(H'_I,j)}\leq C$ for all $j\in C_V(H,I)$}]
\intertext{and, noticing that the last term is zero, this is}
  & \leq \IP[\abs{C(H,1)}>C] + C\IP[\abs{C(H,1)}>C] \\
 & = (C+1)\IP[\abs{C(H,I)}>C]
\ee
Finally, with $\beta' = 2C^5\norm{\D h}/\sigma$,
\ba
  & \IP[\abs{D'}>\beta'] \\
  & \leq \IP[\abs{C(H,I)}>C] + \IP[\abs{D'}>\beta',\abs{C(H,I)}\leq C]\\
  & \leq \IP[\abs{C(H,1)}>C] \\
  & \quad + \IP[\abs{C(H,I)}\leq C, \text{$\exists j\in C_V(H,I)$ s.t.\
$\abs{C(H'_I,j)}>C$}] \\
  &\quad +\IP[\abs{D'}>\beta',\abs{C(H,I)}\leq C,\text{
$\abs{C(H'_I,j)}\leq C$ $\forall j\in C_V(H,I)$}]\\
  & \leq \IP[\abs{C(H,1)}>C] 
  + C\IP[\abs{C(H,1)}>C] \\
  &\quad +\IP\bklel\abs{C(H,I)}\leq C,\text{
$\abs{C(H'_I,j)}\leq C$ $\forall j\in C_V(H,I)$},\\
  &\kern20em\text{$\exists j\in V_I$ s.t.\ $\abs{C(H''_I,j)}>C$}\bkler\\
  &\quad +\IP\bklel\abs{D'}>\beta',\abs{C(H,I)}\leq C,\text{
$\abs{C(H'_I,j)}\leq C$ $\forall j\in C_V(H,I)$},\\
  &\kern20em\text{$\abs{C(H''_I,j)}\leq C$
$\forall j\in V_I$}\bkler
\intertext{and, noticing that the last term is zero, this is}
  & \leq \IP[\abs{C(H,1)}>C] 
  + C\IP[\abs{C(H,1)}>C] + C^2\IP[\abs{C(H,1)}>C]\\
  & \leq (C+1)^2\IP[\abs{C(H,1)}>C].
\ee

We also have the rough bounds
\be
  \abs{G},\abs{D} \leq \frac{2n^2}{\sigma}\norm{h},
\ee

Hence, choosing $C = 8\ln(n\norm{h})/a - 1$ we have that $C\geq1$ from \eq{70},
and using \eq{71} we have
\ba
  \IP[\abs{G}>\alpha] & \leq \frac{e^{-aC}}{\lambda}\leq
    \frac{e^a}{\lambda n^4\norm{h}^2},\\
  \IP[\abs{D}>\beta] & \leq (C+1)\frac{e^{-aC}}{\lambda}\leq
    \frac{4e^a\ln(n\norm{h})}{a\lambda n^4\norm{h}^2},\\
  \IP[\abs{D'}>\beta'] & \leq (C+1)^2\frac{e^{-aC}}{\lambda}\leq
    \frac{16e^a\ln(n\norm{h})^2}{a^2\lambda n^4\norm{h}^2}.
\ee

From Theorem~\ref{thm2} we hence have
\bes
  & \dk(\law(W),\N(0,1))\\
  &\leq 12(\alpha\beta+1)\beta'+8\alpha\beta^2\\
  &\quad+\frac{4n^4\norm{h}^2}{\sigma^2}\bklr{\IP[\abs{G} >
  \alpha]+\IP[\abs{D} > \beta]+\IP[\abs{D'} > \beta']}\\
  &\leq 12\bbbklr{\frac{n\norm{\D
      h}}{\sigma}\bbbklr{C^2+\frac{2}{\lambda
      a^2}}\frac{2C^4\norm{\D
      h}}{\sigma}+1}\frac{2C^5\norm{\D h}}{\sigma}\\
  &\quad+8\frac{n\norm{\D
      h}}{\sigma}\bbbklr{C^2+\frac{2}{\lambda
      a^2}}\frac{4C^8\norm{\D h}^2}{\sigma^2}\\
  &\quad+\frac{4n^4\norm{h}^2}{\sigma^2}\bbbklr{
  \frac{e^2}{\lambda n^4\norm{h}^2}
  + \frac{4e^a\ln(n\norm{h})}{a\lambda n^4\norm{h}^2}
  +\frac{16e^a\ln(n\norm{h})^2}{a^2\lambda n^4\norm{h}^2}
  }\\
  &\leq 24\bbbklr{\bbbklr{C^2+\frac{2}{\lambda
      a^2}}C^4+\frac{\sigma^2}{n}}\frac{nC^5\norm{\D
h}^3}{\sigma^3}\\
  &\quad+32\bbbklr{C^2+\frac{2}{\lambda
      a^2}}\frac{C^8n\norm{\D h}^3}{\sigma^3}
  +\frac{84e^a\ln(n\norm{h})^2}{a^2\lambda\sigma^2}\\
  &\leq
32\bbbklr{2+\frac{4}{a^2\lambda}+\frac{\sigma^2}{n}}\frac{nC^{11
} \norm { \D
h}^3}{\sigma^3}
    +\frac{84e^a\ln(n\norm{h})^2}{\lambda a^2\sigma^2}.
\ee

\end{proof}

\begin{remark}\label{rem2} Note that $(H,H_i')$ does not form an
exchangeable pair, although the marginal distributions are the same. To
see this, denote by $G_c$ the complete
graph and by $G_0$ the empty graph on the vertices $[n]$, where we assume~$n>2$.
Now, given $H=G_c$, $H'_i$ is just an independent
realisation of the \ER\ random graph, hence
\be
  \IP[H_i'=G_0|H=G_c] = \IP[H_i'=G_0] = (1-p)^{n\choose 2}.
\ee
On the other
hand,
\be
  \IP[H=G_0|H_i'=G_c] 
    = \frac{\IP[H_i'=G_c|H=G_0]\IP[H=G_0]}{\IP[H_i'=G_c]}
    = 0
\ee
since it is not possible that $H_i'$ is a complete graph if $H$ is empty.
\end{remark}

\section{Zero bias transformation}

Assume that $\IE W = 0$ and~$\Var W=1$. It was proved in \cite{Goldstein1997}
that there exists a unique distribution $\law(W^z)$ such that, for all smooth
$f$ we have
\ben                                                    \label{72}
    \IE\bklg{Wf(W)} = \IE f'(W^z)
\ee
Furthermore, $\law(W^z)$ has a density $\rho$ with respect to the Lebesgue
measure. Let us first discuss the connection between our general framework
and~$\rho$.

\begin{lemma}\label{lem5} Let $(W,W',G)$ be a Stein coupling. Then, with
$\^K(t)$ as in~\eq{20},
\ben                                                \label{73}
    \IE\bklg{G(\I[W \leq u < W'] - \I[W'\leq u< W])} = \IE \^K(u-W) = \rho(u)  
\ee
for Lebesgue almost all~$u\in\IR$.
\end{lemma}
\begin{proof} The first equality is clear. To prove the second one, let $f$ be a
bounded
Lipschitz-continuous function. We have
\be
    Gf(W') - Gf(W)
     = \int_{\IR} f'(W+t)\^K(t)dt
     = \int_{\IR} f'(u) \^K(u-W)du,
\ee 
so that, from \eq{8} and using Fubini's Theorem, 
\ben                                                                \label{74}
    \IE\klg{Wf(W)} = \int_{\IR} f'(u)\IE \^K(u-W)du.
\ee 
As we may take $f(x) = \delta^{-1}\int_0^\delta \I[x\leq a] \leq 1$ and thus
$f'(x) = \delta^{-1}I[a\leq x\leq a+\delta]$ for any $a\in\IR$ and $\delta>0$,
we have from \eq{72} and \eq{74} that
\be
    \int_{a}^{a+\delta}\IE\^K(u-W)du = \int_{a}^{a+\delta}\rho(u) du
\ee
which proves the claim as $a$ and $\delta$ are arbitrary.
\end{proof}

In \cite{Goldstein1997} and \cite{Goldstein2005a} a method was introduced to
construct $\law(W^z)$ using an underlying exchangeable pair satisfying the
linearity condition \eq{24} with~$R=0$. Although the construction itself does
not directly lead to a coupling of $W^z$ with $W$ (which is what we ultimately
want), it can nevertheless suggest ways to find such couplings; see for example
\cite{Goldstein2005b} and \cite{Ghosh2009}.

We can generalize the idea to our setting. However, as some of the involved
measures may become signed  measures, we have to proceed with more care.
Denote
by $F$ the probability measure on $\IR^2$ induced by~$(W,W')$. With \be
    \phi(w,w') := \IE(GD\,|\,W=w,W=w') = (w'-w)\IE(G\,|\,W=w,W=w') 
\ee
define a new (possibly signed) measure 
\be
    d\^F(w,w') = \phi(w,w')dF(w,w').
\ee
We have $\int_{\IR^2}d \^F(w,w') = 1$ because $\IE(GD) = 1$ from \eq{8}.
Let now the space $\^\Omega := \IR^2\times[0,1]$ be equipped with the standard
Borel $\sigma$-algebra and define the measure $Q = \^F \otimes \ell$ where
$\ell$ is the Lebesgue measure on~$[0,1]$. Define the mapping $W^z:\^\Omega\to
\IR$ as
\be
    W^z(w,w',u) := uw'+(1-u)w
\ee
Clearly, $W^z$ is measurable. We now have the following.

\begin{lemma}\label{lem6} Assume that  \eq{8} holds. Then, the measure $P^z$
on $\IR$ induced by the
mapping $W^z$ is a probability measure and for every function $f:\IR\to\IR$ with
bounded derivative we have
\ben                                                \label{75}
    \IE\bklg{Wf(W)} = \int_{\^\Omega} f'(W^z(w,w',u)) dQ(w,w',u)
        = \int_{\IR}f'(x)\rho(x)dx.
\ee
\end{lemma}
\begin{proof} Let us proof the first equality of \eq{75}.
\ba
    &\int_{\^\Omega} f'(W^z(w,w',u)) dQ(w,w',u)\\
    &\qquad=\int_{\IR^2}\int_{[0,1]} f'(w+u(w'-w)) du\,d\^F(w,w')\\
    &\qquad=\int_{\IR^2}\frac{f(w')-f(w)}{w'-w}d\^F(w,w')\\
    &\qquad=\int_{\IR^2}\phi(w,w')\frac{f(w')-f(w)}{w'-w}d F(w,w')\\
    &\qquad=\IE\{ G(f(W')-f(W))\} = \IE\{Wf(W)\}.
\ee
We thus have proved that the measure $P^z$, induced by $W^z$ and $Q$, satisfies
\ben                                                        \label{76}
    \IE\{Wf(W)\} = \int_{\IR} f'(x) dP^z(x).
\ee
Using the special functions from the proof of Lemma~\ref{lem5}, it is clear
from \eq{76} and \eq{72} that $\rho$ is the Radon-Nykodim derivative of $P^z$
with respect to the Lebesgue measure. This implies the second equality in
\eq{75}.
\end{proof}

In the case where $\phi(W,W')\geq 0$ almost surely, $\^F$ is also a probability
measure and we can write
\be
    W^z = U\^W' + (1-U)\^W
\ee
where $(W,W')$ has distribution $\^F$ and $U\sim U[0,1]$ is independent
of~$(W,W')$.

\section{Proofs of main results} \label{sec14}

\subsection{Preliminaries}

For a random variable $X$ define the truncated version
$\trunc{X}:= (X\wedge 1)\vee(-1)$. For a real number $t$ define $t_+ = t\vee 0$
and $t_- = t\wedge 0$.

Stein's method for normal approximation is based on the differential equation
\ben                                                        \label{77}
    f'(w) - wf(w) = h(w) - \IE h(Z)
\ee
where $Z\sim\N(0,1)$ which can be solved for any measurable function $h$ for
which $\IE h(Z)$ exists. The solution $f_h$ is differentiable and, if $h$ is
Lipschitz, also $f_h'$ is Lipschitz; see \cite{Stein1986}.

\begin{lemma}\label{lem7} Let $W$, $W'$, $W''$, $\~D$ and $G$ be square
integrable random variables. Let $h$ be a measurable function for which $\IE
h(W)$ and $\IE h(Z)$ exist and let $f$ be the solution to \eq{77}. Then
\be                                                    
   \babs{\IE h(W) - \IE h(Z)}\leq (\norm{f}\vee\norm{f'})r_0 +
\norm{f'}(r_1+r_2+r_3) + 
        \abs{\IE R_1(f)}+\abs{\IE R_2(f)},
\ee
where
\ban
    R_1(f) & = (G\~D-S)\klr{f'(W'')-f'(W)}          \label{78}\\
    R_2(f) & = G\int_0^D\bklr{f'(W+t)-f'(W)}.       \label{79}
\ee
\end{lemma}
\begin{proof} Let $f=f_h$ be the solution to \eq{77}. We can assume that
$\norm{f}$ and $\norm{f'}$ are finite otherwise the statement is trivial. From
the fundamental
theorem of calculus, we have
\ben                                                                \label{80}
    f(W') - f(W) = \int_{0}^D f'(W+t) dt.
\ee
Multiplying \eq{80} by $G$ and comparing it with the left hand side of \eq{77}
we have
\besn                                \label{82}
    h(W) - \IE h(W) & = f'(W) - Wf(W)\\
        & = Gf(W')-Gf(W)-Wf(W)  \\
        &\quad + (S-G\~D)f'(W'')\\
        &\quad + (1-S)f'(W)\\
        &\quad + G(\~D-D)f'(W)\\
        &\quad + (S-G\~D)\klr{f'(W)-f'(W'')}\\
        &\quad - G\int_{0}^D \bklr{f'(W+t)-f'(W)} dt.
\ee
Taking expectation, the lemma is immediate.
\end{proof}

\subsection{Bound on the Wasserstein distance}

If $h$ is Lipschitz continuous, then the solution $f$ to \eq{77} is
differentiable, $f'$ is Lipschitz and we have the following bounds:
\ben                                                        \label{83}
    \norm{f} \leq 2\norm{h'}, \quad
    \norm{f'} \leq \sqrt{\frac{2}{\pi}}\norm{h'}, \quad
    \norm{f''} \leq 2\norm{h'}.
\ee (see \cite{Stein1972} and \cite{Raic2004}).
\begin{proof}[Proof of Theorem~\ref{thm1}] The following bounds are easy to
obtain using Taylor's theorem:
\be
    \abs{\IE R_1(f)}  \leq 2 r_4'\norm{f'}+ r_5'\norm{f''},
    \qquad \abs{\IE R_2(f)} \leq 2r_4\norm{f'}+ 0.5 r_5\norm{f''}.
\ee
Combining this with Lemma~\ref{lem7} and the bounds \eq{83} for Lipschitz $h$
with $\norm{h'}\leq 1$
proves the theorem.
\end{proof}

\subsection{Bounds on the Kolmogorov distance}

If, for some $\eps>0$, $h$ is of the form
\ben                                                                \label{84}
    h_{a,\eps}(x) = 
    \begin{cases}
     1,&\text{if $x\leq a$,}\\
     1+(a-x)/\eps,&\text{if $a\leq x\leq a+\eps$,}\\
     0,&\text{if $x>a+\eps$,}\\
    \end{cases}
\ee
then, from pages 23 and 24 of
\cite{Stein1986} (see also \cite{Chen2004a}) we have for $f_{a,\eps}$ for
every $w,v\in\IR$ the bounds
\ben                                                \label{85}
    0\leq f_{a,\eps}(w) \leq \sqrt{2\pi}/4, \qquad
    \abs{f_{a,\eps}'(w)}\leq 1, \qquad
    \abs{f_{a,\eps}'(w)-f_{a,\eps}'(v)}\leq 1
\ee
and, in addition for every $s,t\in\IR$,
\ban
    &\abs{f_{a,\eps}'(w+s)-f_{a,\eps}'(w+t)}\notag\\
    &\qquad\leq(\abs{w}+1)\min(\abs{s}+\abs{t},1) 
    +\eps^{-1}\int_{s\wedge t}^{s\vee t}\I[a\leq w+u\leq a+\eps]du\label{86}\\
    &\qquad\leq(\abs{w}+1)\min(\abs{s}+\abs{t},1) 
    + \I[a-s\vee t\leq w\leq a-s\wedge t+\eps].\label{87}
\ee

Throughout this section let $\kappa:=\dk\bklr{\law(W),\N(0,1)}$.
Now, it is not difficult to see that for any~$\eps>0$,
\ben                                                        \label{90}  
    \kappa
    \leq \sup_{a\in\IR}\babs{\IE h_{a,\eps}(W)-\IE h_{a,\eps}(Z)} + 0.4\eps,
\ee
so that we can use Lemma~\ref{lem7} for functions of the form~\eq{84}.
If $f_{a,\eps}$ is the solution to \eq{77}, from the bounds \eq{85} we thus have
\ben                                                        \label{91}
    \kappa \leq r_0 + r_1 + r_2 + r_3 + 0.4\eps 
       + \sup_{a\in\IR}\abs{\IE R_1(f_{a,\eps})}
        + \sup_{a\in\IR}\abs{\IE R_2(f_{a,\eps})},
\ee
which will be the basis for further bounds on~$\kappa$.

Recall the well known relation between the arithmetic and geometric mean, that
is, for each $x,y>0$ and $\theta>0$ we have
\ben                                                    \label{92}
    \sqrt{xy}\leq \frac{\theta x+\theta^{-1}y}{2},
\ee
which will be used several times. We will also make use of the
following simple lemma in order to implement the recursive approach; see
\cite{Raic2003}.

\begin{lemma}\label{lem8} For any random variable $V$ and for any $a<b$ we have
\ben                                                    \label{93}
    \IP[a\leq V\leq b] \leq \frac{b-a}{\sqrt{2\pi}}+
    2\dk\bklr{\law(V),\N(0,1)}.
\ee
\end{lemma}

\begin{proof}[Proof of Theorem~\ref{thm2}]
Let $f=f_{a,\eps}$ be the solution to \eq{77}, where we will from now on omit
the dependency on $a$ and $\eps$ for better readability. 
Let
$I_1=\I[\abs{G}\leq\alpha,\abs{\~D}\leq\~\beta,\abs{D'}\leq\beta',\abs{S}\leq
\gamma]$ and write
\eq{78} as 
\bes
    \IE R_1(f) & =
    \IE\bklg{(G\~D-S)(1-I_1)(f'(W'')-f'(W))}\\
    &\quad+\IE\bklg{(G\~D-S)I_1(f'(W'')-f'(W))} =: J_1 + J_2.
\ee
Using \eq{85}, the bound $\abs{J_1}\leq r_6'$ is immediate. Let for convenience
$k:=\IE\abs{W}+1$. Using \eq{86} and Lemma~\ref{lem8},
\bes
    J_2 
    &\leq\IE\babs{(G\~D-S)I_1\bklr{f'(W'')-f'(W)}}\\
    &\leq(\alpha\~\beta+\gamma)k\beta'
    +(\alpha\~\beta+\gamma)\eps^{-1}\int_{-\beta'}^{\beta'}
        \IP[a\leq W+u\leq a+\eps]du\\
    &\leq(\alpha\~\beta+\gamma)k\beta'
    +0.8(\alpha\~\beta+\gamma)\beta'
    + 4(\alpha\~\beta+\gamma)\beta'\eps^{-1}\kappa.
\ee

Similarly, let
$I_2=\I[\abs{G}\leq\alpha,\abs{D}\leq\beta]$ and write \eq{79} as
\bes
    \IE R_2(f)  & = \IE\bbbklg{G (1-I_2)\int_0^D\bklr{f'(W+t)-f'(W)}dt}\\
    & \quad+ \IE\bbbklg{GI_2\int_0^D\bklr{f'(W+t)-f'(W)}dt}
    =: J_3 + J_4.
\ee
By \eq{85}, $\abs{J_3}\leq r_6$. Using \eq{86} and Lemma~\ref{lem8},
\bes
    J_4 
    &\leq\IE\bbbabs{GI_2\int_{D_-}^{D_+}\babs{f'(W+t)-f'(W)}dt}\\
    &\leq\alpha\int_{-\beta}^{\beta}\abs{t}k dt
    +\alpha\eps^{-1}\int_{-\beta}^{\beta}\int_{t_-}^{t_+}
       \IP[a\leq W+u \leq a+\eps] dudt\\
    &\leq \alpha\beta^2k 
        +0.4\alpha\beta^2 + 2\alpha\beta^2\eps^{-1}\kappa,
\ee
so that, collecting all the bounds, setting
$\eps=4\alpha\beta^2+8(\alpha\~\beta+\gamma)\beta'$ and making use of \eq{91},
we obtain
\bes
    \kappa &\leq r_0+ r_1 + r_2+r_3 + r_6 + r_6'
        +(\alpha\~\beta+\gamma)(k+0.8)\beta'+(k+0.4)\alpha\beta^2\\
        &\quad +0.4\eps + \bklr{2\alpha\beta^2+
        4(\alpha\~\beta+\gamma)\beta'}\eps^{-1}\kappa\\
        &= r_0+ r_1 + r_2 + r_3 + r_6 + r_6'
        +(\alpha\~\beta+\gamma)(k+4)\beta'+(k+2)\alpha\beta^2 + 0.5\kappa      
\ee
which, solving for $\kappa$, proves the theorem.
\end{proof}

\begin{proof}[Proof of Theorem~\ref{thm3}] Let $f=f_{a,\eps}$ be the solution
to~\eq{77}. As we assume $W''=W$, we have $R_1(f)  = 0$. Write \eq{79} as
\bes
    R_2(f) & = \IE{\int_{-\infty}^\infty \bklr{f'(W+t)-f'(W)}K^W(t)dt}\\
    &= \IE{\int_{\abs{t}>1}\bklr{f'(W+t)-f'(W)}K^W(t)dt}\\
     &\quad+\IE{\int_{\abs{t}\leq1}
            \bklr{f'(W+t)-f'(W)}\bklr{K^W(t)-K(t)}dt}\\
      &\quad +\IE{\int_{\abs{t}\leq1} \bklr{f'(W+t)-f'(W)}K(t)dt}\\
      &=: J_1+J_2 + J_3.
\ee
Clearly, by \eq{85}, $\abs{J_1}\leq r_4$. Now let us bound~$J_2$.
By \eq{87},
\bes
    \abs{J_2}&\leq \IE\int_{\abs{t}\leq 1}\bklr{\abs{W}+1}\abs{t}
                                    \abs{K^W(t)-K(t)}dt\\
       &\quad + \IE\int_0^1 \I[a-t\leq W \leq a+\eps]\abs{K^W(t)-K(t)}dt\\
       &\quad + \IE\int_{-1}^0 \I[a\leq W \leq a-t+\eps]\abs{K^W(t)-K(t)}dt\\
       &=: J_{2,1} + J_{2,2} + J_{2,3},
\ee
Using \eq{92}, we have for any $\theta>0$ that
\bes
    J_{2,1} &\leq \frac{\theta}{2}\IE\int_{\abs{t}\leq1}(\abs{W}+1)^2\abs{t}dt
        + \frac{1}{2\theta}\IE\int_{\abs{t}\leq1}\abs{t}\abs{K^W(t)-K(t)}^2dt\\
        &= \frac{\theta}{2}\IE(\abs{W}+1)^2
            + \frac{1}{2\theta}r_8^2.
\ee
Let $\alpha_0 = (\IE(\abs{W}+1)^2)^{1/2}$ and choose $\theta =
\alpha_0^{-1}r_8$,
so that
\be
    J_{2,1} \leq \alpha_0  r_8.
\ee
Let now $\delta = 0.4\eps+2\kappa$. Then, from Lemma~\ref{lem8}, we have for
$t>0$
\ben                                                                \label{94}
    \IP[a-t\leq W\leq a+\eps] \leq \delta + 0.4 t.
\ee 
Using \eq{92},
\bes 
    J_{2,2} & \leq \IE\bbbklg{\int_0^1 \bbklel
        0.5\eps\bklr{\delta+0.4t}^{-1}\I[a-t\leq W\leq a+\eps]\\
        &\kern7em +0.5\eps^{-1}\bklr{\delta+0.4t}
            \bklr{K^W(t)-K(t)}^2
    \bbkler dt} \\
    & \leq 0.5\eps + 0.5\eps^{-1}\delta\int_0^1\Var(K^W(t)) dt + 
        0.2 \eps^{-1}\int_0^1 t\Var(K^W(t))dt.     
\ee
A similar bound holds for $J_{2,3}$ so that 
\be
    \abs{J_2}
    \leq \alpha_0r_8 +\eps + 0.5\eps^{-1}\delta r_7 +
0.2\eps^{-1}r_8^2.
\ee
By \eq{86},
\bes
    \abs{J_3} &\leq \IE\int_{\abs{t}\leq 1} (\abs{W}+1)\abs{tK(t)}dt\\
    &\quad+\eps^{-1}\int_{\abs{t}\leq1}\bbbabs{
        \int_0^t\IP[a\leq W+u\leq a+\eps]du
}\abs{K(t)}dt\\
    &\leq (\IE\abs{W}+1)r_5 + \eps^{-1}\int_{\abs{t}\leq 1} \delta\abs{tK(t)}dt
    \leq (\IE\abs{W}+1)r_5 + \eps^{-1}\delta r_5.
\ee
Choose now
\be
    \eps = \sqrt{1.4}\bkle{\kappa(2r_5 + r_7) + 0.2r_8^2}^{1/2}.
\ee
From \eq{91} (note that $r_1=0$ because $\~D = D$), using that
$\sqrt{x+y}\leq \sqrt{x}+\sqrt{y}$ and \eq{92}, we obtain
\bes
    \kappa&\leq r_0 + r_2 + r_3 + 0.4\eps + \abs{J_1}+\abs{J_2}+\abs{J_3}\\
    &\leq r_0 + r_2 + r_3 + r_4 + (\IE\abs{W}+1)r_5 +\alpha_0r_8+1.4\eps\\
    &\quad + \eps^{-1}(\delta(r_5+0.5r_7)+0.2r_8^2)\\
    &\leq r_0 + r_2 + r_3 + r_4 + (\IE\abs{W}+1.4)r_5 
        + 0.2r_7 + \alpha_0r_8 + 1.4\eps \\
    &\quad + \eps^{-1}\bklr{\kappa(2r_5 + r_7) + 0.2r_8^2}\\
    &\leq r_0 + r_2 + r_3 + r_4 + (\IE\abs{W}+1.4)r_5 
        + 0.2r_7 + \alpha_0r_8  \\
    &\quad + 2.4\bklr{\kappa(2r_5 + r_7) + 0.2r_8^2}^{1/2}\\
    &\leq r_0 + r_2 + r_3 + r_4 + (\IE\abs{W}+1.4)r_5 
        + 0.2r_7 + (\alpha_0+1.1)r_8  \\
    &\quad + 2.4\bklr{\kappa(2r_5 + r_7)}^{1/2}\\
    &\leq r_0 + r_2 + r_3 + r_4 + (\IE\abs{W}+1.4+2.4\theta^{-1})r_5 
        + (0.2+1.2\theta^{-1})r_7  \\
    &\quad  + (\alpha_0+1.1)r_8+ 1.2\theta\kappa\\
\ee
Choosing $\theta = 1/2.4$ and solving for $\kappa$ proves the claim.
\end{proof}

\begin{proof}[Proof of Theorem~\ref{thm4}]
Let $f=f_{a,\eps}$ be the solution to \eq{77}. Now, from~\eq{86},
\bes
    \abs{R_1(f)} 
    &\leq  \babs{(S-G\~D)(\abs{W}+1) \-D'}\\
    &\quad + \abs{S-G\~D}\eps^{-1} 
        \int_{D'_-}^{D'_+}\I[a\leq W + u \leq a + \eps] du
\ee
so that 
\bes
    \IE\abs{R_1(f)} 
    & \leq \IE\babs{(S-G\~D)(\abs{W}+1) \-D'}\\
    &\quad+\eps^{-1}\IE\babs{(S-G\~D)D' S_\eps(\law(W|G,\~D,D'))}.      
\ee
Furthermore,
\bes
    \abs{R_2(f)} & 
    \leq \abs{G}\int_{D_-}^{D_+}(\abs{W}+1)(\abs{t}\wedge 1)dt\\
    &\quad +\eps^{-1}\abs{G}\int_{D_-}^{D_+}
    \int_{t_-}^{t_+}\I[a\leq W+u\leq a+\eps]dudt
\ee
hence
\be
    \IE\abs{R_2(f)}\leq 0.5\IE\abs{G(\abs{W}+1)\-D^2}
    + 0.5\eps^{-1}\IE\babs{GD^2 S_\eps(\law(W|G,D))}
\ee
Combining these bounds with  \eq{91} proves the theorem. 
\end{proof}

To prove Lemma~\ref{lem1} we first need a simple lemma.

\begin{lemma}\label{lem12} Let $\beta_{k,l}$, $1\leq l< k$,
$k=2,\dots,n$, be non-negative real numbers. If $a_1 \leq b_1 = 1$ and if
there are constants $q\geq 1$ and $p<1$ such
that, for all $k=2,3,\dots,n$, we have $\sum_{l=1}^{k-1}\beta_{k,l}\leq p$,
\be
    a_k = q + \sum_{l=1}^{k-1}\beta_{k,l}a_l\qquad\text{and}\qquad
    b_k = q + pb_{k-1},
\ee
then $a_k\leq b_k \leq q/(1-p)$
for all $1\leq k\leq n$.
\end{lemma}

\begin{proof} Note first that $q/(1-p)\geq 1$ and
\be
    b_k = \bbbklr{1-\frac{q}{1-p}}p^k + \frac{q}{1-p},
\ee
hence $b_1,b_2,\dots$ is increasing with upper bound~$q/(1-p)$. The proof of
$a_k\leq b_k$ is now a simple induction on~$k$. By assumption $a_1\leq b_1$,
which verifies the base case. Using that $a_l\leq b_l$ for all $1\leq l\leq k$,
we have
\bes
    a_{k+1} & = q + \sum_{l=1}^{k}\beta_{k+1,l}a_l
    \leq q + \sum_{l=1}^{k}\beta_{k+1,l}b_l\\
    & \leq q + \sum_{l=1}^{k}\beta_{k+1,l}b_{k} 
    \leq q + pb_{k} = b_{k+1}. \qedhere
\ee
 
\end{proof}

\begin{proof}[Proof of Lemma~\ref{lem1}] First note that \eq{22} still holds if
we replace $A$ by $\-A :=
A\vee 1$. Let $\sigma_1:=1$ and define a new sequence $a_k = \sigma_k\kappa_k$
for $1\leq k \leq n$. Using inequality \eq{22} with $\eps =
c_n\alpha_n/\sigma_k$ for a constant $c_n>1$ to be chosen later,
\be
    a_k \leq \-A + 0.4c_n\alpha_n +
    \sum_{l=1}^{k-1} \beta_{k,l} a_l
\ee
for all $1\leq k\leq n$, where 
\be
    \beta_{k,l} = \frac{\sigma_kA_{k,l}}{c_n\sigma_{l} \alpha_n}.
\ee
Note that $\sum_{l=1}^{k-1}\beta_{k,l} \leq c_n^{-1}$.
Consider now the solution $b_k$ to the recursive equation
\be
    b_1 = 1; \qquad b_k = \-A + 0.4 c_n\alpha_n +
c_n^{-1}b_{k-1},\quad\text{for $2\leq k\leq n$.}
\ee
Note that $\-A + 0.4 c_n\alpha_n\geq 1$, hence we obtain
from Lemma~\ref{lem12} that
\be
    a_n\leq b_n \leq \frac{c_n(\-A+0.4c_n\alpha_n)}{c_n-1}.
\ee
Minimizing over $c_n>1$ we can chose 
\be
    c_n = 1 + \frac{\sqrt{2\alpha_n(2\alpha_n+5\-A))}}{2\alpha_n} 
    = 1+\frac{\alpha_n'}{2\alpha_n}.
\ee
Recalling that $\kappa_n = a_n/\sigma_n$, the claim follows.
\end{proof}

\section*{Acknowledgments}

The authors would like to thank Andrew Barbour, Omar El-Dakkak, Larry
Goldstein, Erol Pek\"oz and Nathan Ross for helpful discussions.

% \bibliographystyle{mynatbib}
% \bibliography{literatur}

\end{document}